\documentclass[11pt]{amsart}

\usepackage[margin=1.15in]{geometry}

\usepackage{cancel}
\usepackage{amssymb}
\usepackage{enumerate}
\usepackage{hyperref}
\usepackage{enumitem}
\usepackage{amsmath}
\usepackage{latexsym}
\usepackage{extarrows}
\usepackage{graphicx}
\usepackage{txfonts}
\usepackage{mathtools}
\usepackage{booktabs}
\usepackage{float}
\usepackage{bm}
\usepackage[normalem]{ulem}
\usepackage{soul}
\usepackage{tikz-cd}
\usepackage{xcolor}
\usepackage{quiver}
\usepackage{comment}
\usepackage{graphics}
\usepackage{float}
\usepackage{relsize}
\usepackage{bm}
\usepackage{graphicx}
\usepackage{tikz-cd}

\usepackage{xcolor}
\usepackage{todonotes}
\usepackage{comment}

\usepackage[
  backend=biber,
  style=numeric,
  maxnames=10
]{biblatex}
\numberwithin{equation}{section}

\theoremstyle{plain}
\newtheorem{thm}{Theorem}[section]
\newtheorem{prop}[thm]{Proposition}
\newtheorem{lemma}[thm]{Lemma}
\newtheorem{cor}[thm]{Corollary}
\newtheorem{conj}[thm]{Conjecture}
\newtheorem{question}{Question}

\theoremstyle{definition}
\newtheorem{dfn}[thm]{Definition}
\newtheorem{ex}[thm]{Example}
\newtheorem{remark}[thm]{Remark}
\newtheorem*{thm*}{Theorem}

\theoremstyle{remark}

\newcommand{\C}{\mathbb{C}}
\newcommand{\Q}{\mathbb{Q}}
\newcommand{\Z}{\mathbb{Z}}
\newcommand{\N}{\mathbb{N}}
\newcommand{\E}{\mathbb{E}}
\newcommand{\XX}{\mathrm{X}}

\newcommand{\inner}[2]{\langle #1, #2 \rangle}
\newcommand{\ind}[1]{\mathbf{1}_{#1}}

\newcommand{\mps}{(X,\mathcal{B},\mu,T)}

\newcommand{\IP}{\mathrm{IP}}
\newcommand{\fol}{(\Phi_N)_{N\in \N}}
\renewcommand{\ni}{(n_i)_{i\in\N}}
\newcommand{\IPr}{\mathrm{IP}_{\mathrm{rat}}}

\newcommand{\avgip}{\lim_{N\rightarrow \infty} \E_{n\in \IP_{\Phi_N}(\ni)}}

\newcommand{\rat}{\mathrm{rat}}
\newcommand{\Krat}{K_{\mathrm{rat}}}

\title[Polynomial recurrence along IP sets]
{$\IPr$-Polynomial recurrence and large intersections}
\author{Borys Holikov}
\address{Department of Mathematics, Bar-Ilan University, Ramat Gan, Israel}
\email{holikob@biu.ac.il}
\author{Or Shalom}
\address{Department of Mathematics, Bar-Ilan University, Ramat Gan, Israel}
\email{Or.shalom@math.biu.ac.il}

\date{}

\begin{document}

\begin{abstract}
Let $p_1,...,p_k$ be a rationally independent sequence of integer valued polynomials. We show that for all $E\subseteq \mathbb{N}$, every F\o lner sequence $\Phi$, and every $\varepsilon>0$, the set
$$\left\{n\in \mathbb{N} : d_{\Phi}\left(E\cap(E+p_1(n))\cap\cdots\cap  (E+p_k(n))\right) > d_{\Phi}(E)^{k+1}-\varepsilon\right\}$$ intersects every $\IP$ generated by a sequence with rational spectrum. Our methods involve the study of the characteristic factors for multiple ergodic polynomial averages along IPs (see Theorem~\ref{polychar}). In particular, we also prove a pointwise convergence theorem for polynomial averages along IPs with rational spectrum (see Theorem~\ref{bfromleibman}), generalizing a well known result of Leibman from \cite{leibmanpointwise}.

\end{abstract}

\maketitle

\section{Introduction}
\subsection{Motivation and large intersections}
Furstenberg's seminal proof of Szemer\'edi's theorem established a profound connection between ergodic theory and additive combinatorics. 

Within this ergodic framework, Furstenberg established the following multiple recurrence theorem \cite{furstenberg1977ergodic}:
\begin{thm}[Furstenberg multiple recurrence theorem]
For any system $\mps$, any set $A \in \mathcal{B}$ with $\mu(A) > 0$, and any $k\in \N$, there exists a constant $c>0$ such that the set of return times 
$$
\left\{n\in \N : \mu(A\cap T^{-n}A\cap \dots\cap T^{-kn}A) \ge c\right\}
$$
is syndetic \textup{(}i.e., it has bounded gaps\textup{)}.
\end{thm} 

It was shown in \cite{MR1784213} by Bergelson et al., that a value of the constant $c$ which depends only on $\mu(A)$ and $k$ exists. The optimal combinatorial scenario is referred to as the \textit{large intersection property}, which occurs when $c$ can be chosen arbitrarily close to the theoretical limit of independent intersections, specifically $\mu(A)^{k+1}-\varepsilon$ for any $\varepsilon > 0$. Relying on the Host and Kra \cite{hostkranil} and Ziegler \cite{ziegler} structural theories, in \cite{largelinear} Bergelson, Host, and Kra demonstrated that the large intersection property holds for linear configurations of short lengths in ergodic systems:

\begin{thm}[Large intersection property along short arithmetic progressions]
Let $\mps$ be an ergodic measure preserving system, let $A \in \mathcal{B}$ with $\mu(A) > 0$, and let $k=2 \text{ or } 3$. Then for any $\varepsilon>0$, the set 
$$
\{n\in \N : \mu(A\cap T^{-n}A\cap \dots\cap T^{-kn}A) \ge \mu(A)^{k+1}-\varepsilon\}
$$
is syndetic.
\end{thm}

They furthermore showed that for $k \ge 4$, the set of large intersections may be empty. To extend these structural properties beyond linear progressions, one considers polynomial sequences. Specifically, we study \textit{rationally independent integer polynomials}, which are families of polynomials such that no non-trivial linear combination of them over the rationals is a constant. For such families, Frantzikinakis and Kra \cite{indepoly} established that the large intersection property is fully recovered even without the ergodicity assumption:

\begin{thm}[Large intersections for polynomial sequences]\label{largeFK}
Let $k\geq 1$ be arbitrary. Let $\mps$ be a measure preserving system, $A\in \mathcal{B}$ with $\mu(A)>0$, and $p_1,\dots,p_k$ be a family of rationally independent, integer-valued polynomials with $p_i(0) = 0$ for all $i = 1,\dots, k$. Then for any $\varepsilon>0$, the set 
\begin{equation}\label{polylargeintset}
\left\{n\in \N : \mu(A\cap T^{p_1(n)}A\cap\dots \cap T^{p_k(n)}A) > \mu(A)^{k+1}-\varepsilon\right\}
\end{equation}
is syndetic.
\end{thm}
Note that here and in other results of this paper, it is impossible to get a similar result without the $\varepsilon$ error (cf., \cite{underrecurrent}). \\
Multiple generalizations of this theorem that are not directly related to this paper exist, see e.g., \cite{ABB,shalom2,ABS,shalomvectorspace} for the linear case in general countable abelian groups, \cite{AckelsbergBergelson} for a version of this result in number fields, \cite{BLcubic} for cubic averages, \cite{FranKuca} for several commuting transformations, and \cite{moreexamples,Hardyexamples,jerg} for Hardy sequences. Some of these assume additional conditions on the system.

\subsection{$\IP$ configurations and sequences with rational spectrum}

Another fundamental approach to characterizing the ``largeness'' of a subset of $\N$ relies on its intersections with structured families of sets. Given an infinite sequence of natural numbers $\ni$, in the literature the \textit{$\IP$} generated by it, denoted $\IP(\ni)$, is defined as the set of all finite sums of distinct elements from the sequence. However, here we will need a somewhat more general definition (see Definition \ref{incfol}). Moreover, a subset of $\N$ is \textit{$\IP^*$} if it intersects every $\IP(\ni)$ non-trivially.

In \cite{ipszemeredi}, Furstenberg and Katznelson proved the $\IP$-Szemer\'edi theorem, yielding the following strong recurrence property:

\begin{thm}
For any ergodic system $\mps$, any set $A \in \mathcal{B}$ with $\mu(A) > 0$, and any $k\in \N$, the set 
$$
\{n\in \N : \mu(A\cap T^{-n}A\cap \dots\cap T^{-kn}A) > 0\}
$$
is $\IP^*$. 
\end{thm}

The analogous result for sequences for rationally independent polynomials was obtained in \cite{polyipszemeredi} by Bergelson and McCutcheon:
\begin{thm}
Let $\mps$ be an ergodic system and let $A\in \mathcal{B}$ with $\mu(A) > 0$. For any $k\in \N$ and any family of rationally independent integer polynomials $p_1,\dots,p_k$ satisfying $p_i(0) = 0$ for all $i=1,\dots, k$, the set 
$$
\{n\in \N : \mu(A\cap T^{p_1(n)}A\cap \dots\cap T^{p_k(n)}A) > 0\}
$$
is $\IP^*$. 
\end{thm}

To refine these intersection properties, we restrict our attention to $\IP$ sets generated by sequences $\ni$ possessing a \textit{rational spectrum} (see Definition~\ref{ratspec}).

The second author and Kra \cite{krashalom2025} recently demonstrated that for linear configurations, one can simultaneously achieve the large intersection property and positive lower $\IPr$ density (see Definition \ref{ipratden}) along $\IP$s with rational spectrum:
\begin{thm}\label{KraShalom}
Let $\mps$ be a measure preserving system and $A\in \mathcal{B}$ with $\mu(A)>0$. Then for all coprime integers $l_1,l_2\in \Z$ and any $\varepsilon >0$, the sets
$$
\{n\in \N : \mu(A\cap T^{-l_1n}A\cap T^{-l_2n}A)\ge \mu(A)^3-\varepsilon\}
$$
and 
$$
\{n\in \N : \mu(A\cap T^{-l_1n}A\cap T^{-l_2n}A\cap T^{-(l_1+l_2)n}A)\ge \mu(A)^4-\varepsilon\}
$$
have positive lower $\IPr$-density.
\end{thm}
Crucially, the assumption of rational spectrum in Theorem~\ref{KraShalom} is a necessary constraint (see \cite[Appendix B]{krashalom2025}). It is natural to ask whether the results of Frantzikinakis and Kra on polynomial progressions (Theorem~\ref{largeFK}) can be extended to IPs. The case of a single polynomial $\mu(A\cap T^{p_1(n)}A)$ was established in \cite{bergelsonupdate}, see also \cite{BFM,bergelson2003minimal,BKM,bergelsonultra} for several related results. However, a recent paper of Bergelson and Zelada \cite{bergelson2026setslargevaluespolynomial} proved that the set \eqref{polylargeintset} is not $\IP^*$ in general, whenever at least two polynomials are involved. But they showed that it is  always $\text{A-}\IP^*$ (almost $\IP^*$), so that it differs from some $\IP^*$ set by a set of zero density.

\subsection{Main results} 

The principal objective of this paper is to establish the $\IPr$ large intersection property for families of rationally independent polynomials. In fact, we also show that the set of differences with this property has positive lower $\IPr$-density.

\begin{thm}[Large intersections for polynomials along rational IPs]\label{main:thm}
Let $\mps$ be an invertible measure preserving system, $k \ge 1$, and let $p_1,\dots,p_k : \Z \rightarrow \Z$ be a family of rationally independent polynomials. Assume furthermore that $p_i(0) = 0$ for all $i = 1,\dots,k$. Then for any $A\in \mathcal{B}$ with $\mu(A) >0$ and all $\varepsilon > 0$, the set 
\begin{equation}\label{maininmain}
\left\{n\in \N: \mu(A\cap T^{p_1(n)}A\cap T^{p_2(n)}A\cap\dots\cap T^{p_k(n)}A) > \mu(A)^{k+1} - \varepsilon\right\} 
\end{equation}
has positive lower $\IPr$-density.  
\end{thm}

In order to translate this dynamical result into a purely combinatorial statement, we rely on the notion of \textit{upper density} with respect to a F\o lner sequence $\Phi=(\Phi_N)_{N\in \mathbb{N}}$, denoted $d_{\Phi}(\Lambda)\colon= \limsup_{N\rightarrow\infty} \frac{|\Lambda\cap \Phi_N|}{|\Phi_N|}$, which measures the maximum asymptotic density of a set $\Lambda \subseteq \N$ over the F\o lner sequence. Since Theorem \ref{main:thm} does not require the system to be ergodic, applying the Furstenberg correspondence principle with Theorem~\ref{main:thm} yields the following combinatorial application.

\begin{cor}\label{main:cor}
Let $\Phi=(\Phi_N)_{N\in \N}$ be a F\o lner sequence, and let  $\Lambda \subset \N$ with $d_\Phi(\Lambda)>0$. Let
$p_1,\ldots,p_k:\Z \to \Z$ be a family of rationally independent polynomials such that $p_i(0)=0$ for all $i=1,\ldots,k$. Then for every
$\varepsilon>0$, the set
$$
\left\{
n \in \N :
d_{\Phi}\bigl(
\Lambda \cap (\Lambda+p_1(n)) \cap (\Lambda+p_2(n)) \cap \cdots
\cap (\Lambda+p_k(n))
\bigr)
>
d_{\Phi}(\Lambda)^{k+1}-\varepsilon
\right\}
$$
has positive lower $\IPr$-density.
\end{cor}

\subsection{Open questions}

A consequence of a recent preprint by Bergelson and Zelada \cite{bergelson2026setslargevaluespolynomial} is that the large intersection set defined in \eqref{maininmain} is not an $\IP^*$ set, meaning it is not guaranteed to intersect every possible $\IP$. This observation motivates the following structural question:

\begin{question}
For which sequences $\ni$ with an irrational spectrum the set in \eqref{maininmain} is always guaranteed to intersect $\IP(\ni)$?
\end{question}

A further inquiry is whether the requirement that all polynomials be strictly nonlinear can be relaxed in Theorem~\ref{polychar}. Specifically, one might ask if the theorem holds for an exceptional polynomial system where exactly one polynomial is linear. We conjecture that the answer is affirmative. As preliminary evidence supporting a positive resolution, we have manually verified that the statement holds for the specific polynomial system $\{n, n^2\}$ by adapting the proof of Furstenberg and Weiss from \cite{MR1412607} to $\IPr$ averages. We thus conjecture
\begin{conj}\label{IPratchar:conj}
    Let $P=(p_1,\dots,p_r)$ be a system of rationally independent polynomials. Then there exists some $k=k(P)$ such that $Z_k(X)$ is an $\IPr$-characteristic factor for $P$ (see Definition~\ref{ipratchar}).
\end{conj}

Other promising directions for future research involve extending these results to encompass more general abelian or nilpotent group actions. Even within the context of $\Z^d$-actions, the analysis of characteristic factors is significantly more intricate. This complexity stems from the fact that unlike standard F\o lner sequences, the product of $d$ distinct $\IP$-F\o lner sequences in $\Z$ does not constitute an $\IP$-F\o lner sequence in $\Z^d$, the former is utilized by Leibman \cite{polycharfactor}. In a broader sense, one might also investigate systems of multiple commuting transformations, analogous to the framework presented in \cite{bergelson2026setslargevaluespolynomial}.
\section*{Acknowledgements}
The second author is supported by Alon Fellowship. This paper is a part of the first author's Master's thesis.

\section{Preliminaries}

\subsection{Measure preserving systems and factors}
An \textit{invertible measure preserving system} (or just a \textit{system} for short) is a quadruple $\mps$, where $X$ is a compact metric space, $\mathcal{B}$ is its Borel $\sigma$-algebra, $\mu$ is a regular probability measure, and $T:X\rightarrow X$ is a $\mu$-preserving invertible transformation (i.e., $\mu(T^{-1}A)=\mu(A)$ for all $A\in \mathcal{B}$). 

We write $L^2(X)$, or $L^2(\mu)$, for the corresponding Hilbert space of all square $\mu$-integrable functions on $X$ modulo equivalence $\mu$-almost everywhere. We abuse the notation and write $T$ both for the measure preserving transformation on $X$ and for an associated unitary operator $T: L^2(X)\rightarrow L^2(X).$

Given two systems $\mps$ and $(Y,\mathcal{C},\nu,S)$, we say that $Y$ is a \textit{factor} of $X$ if there is a measure preserving map $\pi:X\rightarrow Y$, such that $\pi \circ T = S\circ \pi$. If $\pi$ is invertible the systems are called \textit{isomorphic}. Each factor $Y$ defines a unique $T$-invariant sub $\sigma$-algebra $\pi^{-1}(\mathcal{C})\subseteq \mathcal{B}$, and an embedding $L^2(Y)\hookrightarrow L^2(X)$ (defined by $g \mapsto g\circ \pi$). In fact each $T$-invariant sub $\sigma$-algebra of $X$ give raise to a factor. So we will interchangeably think of factors of a system in these two equivalent ways, and sometimes may call a factor the subspace $L^2(Y) \subseteq L^2(X)$, since $Y$ can be reconstructed from $L^2(Y)$. If $Y$ is a factor of $X$, we denote by $E(\cdot \space | Y)$ the conditional expectation with respect to corresponding sub $\sigma$-algebra.

\subsection{Characteristic factors and Nilmanifolds} In this paper we focus on some special families of factors. The first one is the rational Kronecker factor. A function $f:X\rightarrow \C$ is called an \textit{eigenfunction} for $X$ or $T$, if there exists some \textit{eigenvalue} $\lambda\in S^1$ so that $Tf=\lambda f$.
\begin{dfn}
Let $\mps$ be a measure preserving system.  \textit{The  rational Kronecker factor} of $X$ is the minimal sub $\sigma$-algebra of $\mathcal{B}$ with respect to which all eigenfunctions with rational eigenvalues are measurable. We denote it by $K_{\rat}(X)$.
\end{dfn}
 For each $r\in \N$, let $K_r(X)$ denote the factor of $X$ associated with the minimal sub $\sigma$-algebra with respect to which all eigenfunctions with eigenvalues of order $r$ are measurable. It is classical that $K_r(X)$ is isomorphic to finite rotation on a subgroup of $\Z/r\Z$. It is also classical that the rational Kronecker factor is the joining of these $K_r(X)$, namely $K_{\rat}(X) = \bigvee_{r\in \N} K_r(X)$ (Equivalently, $\Krat(X) = \varprojlim K_r(X)$ if we think of factors as a systems and not as sub $\sigma$-algebras).

In \cite{hostkranil} Host and Kra studied the convergence of multiple ergodic averages $\frac{1}{N}\sum_{n=1}^N T^{n}fT^{2n}f\cdots T^{kn}f$. They showed that there exists a sequence of factors $Z_k(X)$ which are \emph{characteristic factors} for multiple ergodic average, and are also inverse limits of nilsystems. Nowadays, we know that these factors are also characteristic for polynomial ergodic averages. Our first goal in this paper is to prove that the Host--Kra factors are $\IPr$-characteristic (see Definition \ref{ipratchar}) for our polynomial averages.

%Nilsystems, namely translations on nilmanifolds, are particularly important for the study of multiple ergodic averages. Let us define precisely what is a nilmanifold and nilsystem.
\begin{dfn}[Nilmanifolds and nilsystems]
Let $G$ be a $k$-step nilpotent Lie group and $\Gamma\le G$ be its discrete co-compact subgroup. Then the homogeneous space $X= G/\Gamma$ is called a $k$-step \textit{nilmanifold}. Together with its Haar measure, $X$ can be seen as a measure preserving system with respect to an action by translation by some $a\in G$. Such a system is called a \textit{nilsystem}. 
\end{dfn}

We let $G_k\le G_{k-1}\le...\le G_2 = [G,G] \le G_1 = G$ denote the lower central series of $G$. We can  immediately see that $G/G_k\Gamma$ is a $(k-1)$-step nilmanifold and a factor of $G/\Gamma$ via $g\Gamma \mapsto gG_k\Gamma$ (we will need this for a proof of a Theorem \ref{pointwise}).

We are now ready to define precisely the Host--Kra structure theorem for $Z_k(X)$. 
\begin{thm}[\cite{hostkranil}]
Let $\mps$ be an ergodic measure preserving system. For each $k\in \N$, the factor $Z_k(X)$ is  an inverse limit of $k$-step ergodic nilsystems.
\end{thm}

We will extensively use the fact that given any $\varepsilon>0$ and $f\in L^2(\XX)$, we can find a $k$-step nilmanifold $G/\Gamma$, such that $f$ is $\varepsilon$-close to some function in $L^2(G/\Gamma)\subseteq L^2(\XX)$ in the $L^2$-norm.

\subsection{Infinite dimensional parallelepipeds}\label{introtoip}
$\IP$ is an abbreviation for infinite-dimensional parallelepipeds. In this section, we define this notion and isolate a special type of $\IP$ possessing a \textit{rational spectrum}, for which we state a mean ergodic theorem. Most of the definitions and results in this section were originally introduced in \cite{krashalom2025}.

\begin{dfn}\label{incfol}
Let $\ni$ be a sequence of natural numbers. The $\IP$ generated by $\ni$ is the multiset $\IP(\ni)$ of all finite sums of elements of $\ni$. Namely,
$$\IP(\ni) = \left\{\sum_{j=1}^{\ell} n_{i_j} : \ell\ge 0, i_1,...,i_{\ell} \text{ are distinct.}\right\},$$ with the convention that the empty sum is zero.
\end{dfn}
A key example is the set of all integers that can be written using only the digits $0,1$ in the decimal basis. 
$$
\IP((10^{n-1})_{n\in\N}) = \{0,1,10,11,100,...\}.
$$

Our next goal is to define ergodic averages along a given $\IP$.
\begin{dfn} 
A F\o lner sequence $\Phi = \fol$ is a sequence of finite subsets of $\N$ satisfying 
$$
\lim_{N\rightarrow\infty}\frac{|(a+ \Phi_N)\Delta \Phi_N|}{|\Phi_N|} = 0
$$
for all $a\in \N$. We say that the F\o lner sequence is \textit{increasing} if $\Phi_N = [M,a_N]$ for some $M\in \N$ and some increasing sequence of natural numbers $(a_n)_{n\in\N}$. Moreover, for all $M\in \mathbb{N}$, we denote by $\Phi^M=(\Phi_N^M)_{N\in\mathbb{N}}$ where $\Phi_{N}^M = \Phi_N \backslash \Phi_M$, the F\o lner sequence obtained by omitting the elements in $\Phi_M$ from every set in $\Phi$.
\end{dfn}
The definition of $\Phi_{N}^M$ is motivated by the $\IP$ van der Corput lemma (see Lemma \ref{vdc}). Moreover, we need to restrict our attention to  increasing F\o lner sequences, because for those the $\IP$-ergodic averages we are about to introduce converge. 

\begin{ex}
 The simplest example of an increasing $\Phi$ is $\Phi_N = [1,N]$. Then $\Phi_N^M = [M+1,N]$.
\end{ex}

Next, we \textit{project} the F\o lner sequence into the $\IP$ in the following manner:
\begin{dfn}
Let $\ni$ be a sequence of natural numbers. For any F\o lner sequence $\Phi$ and $N\in \N$ define 
$$
\IP_{\Phi_N}(\ni) = \{n_{i_1}+n_{i_2}+\dots +n_{i_\ell} : \ell\ge 0, \space i_1,\dots,i_\ell\in \Phi_N \text{ are distinct} \}.
$$

We sometimes write $F_N$ for $\IP_{\Phi_N}(\ni)$ as well as $F_N^M$ for $\IP_{\Phi_N^M}(\ni)$ when F\o lner sequence and $\ni$ are clear from the context.
\end{dfn}
We cautiously note that an $\IP$-F\o lner sequence is in general not a F\o lner sequence, but it can still be used to define ergodic averages along the $\IP$.
\begin{dfn}
Given a sequence of integers $\ni$, a F\o lner sequence $\Phi = \fol$, an integer $k\in \N$ and a sequence $(a_\mathbf{n})_{\mathbf{n} \in \N^k}$ we use the notation 
$$
\E_{n^1,n^2,\dots,n^k\in \IP_{\Phi_N}(\ni)}a_{\mathbf n} :=  \frac{1}{|\IP_{\Phi_N}(\ni)|^k}\sum_{n^1,\dots,n^k\in \IP_{\Phi_N}(\ni)}a_{n^1,\dots,n^k}.
$$
\end{dfn}
In general, ergodic averages along IPs do not need to converge (see \cite{krashalom2025}). However, when we restrict our attention to IPs of \textit{rational spectrum}, they always do.
\begin{dfn}[Spectrum of a sequence]\label{ratspec}
The \textit{spectrum} $\sigma((n_j)_{j\in\N})$ of a sequence of natural numbers $(n_j)_{j\in\N}$ is defined to be the group generated by the complement of the set
$$
\bigcap_{m\in\N}
\left\{
\alpha \in \mathbb{T} :
\lim_{d\to\infty}
\prod_{j=m}^{d}
\left(\frac{1+\alpha^{n_j}}{2}\right)
=0
\right\}.
$$
We say that a sequence has \textit{rational spectrum} if
$\sigma((n_j)_{j\in\N}) \subset e^{2\pi i\Q}$.
\end{dfn}

\begin{ex}
The sequence $n_i = 10^{i-1}$ has a rational spectrum, as well as any sequence $n_i = a^{i-1}$ and most of the naturally constructed sub-exponentially growing sequences. But, for example, any sequence that grows super-exponentially has an irrational spectrum (see., \cite{bergelson2014rigidity}), however it is possible to construct a sequence that has polynomial growth and an irrational spectrum.
\end{ex}

The following mean ergodic theorem along IP-iterates was established in \cite{krashalom2025}.
\begin{thm}[$\IP$ mean ergodic theorem]\label{ipmeanthm}
Let $\mps$ be a measure preserving system, let $\ni$ be a sequence with rational spectrum, and let $\Phi=(\Phi_N)_{N\in\mathbb{N}}$ be an increasing F\o lner sequence. Then 
$$
\avgip T^nf = \sum_{t\in \sigma(\ni)}\omega_{\Phi}(t)P_t(f)
$$
in $L^2(X)$, where $\omega_{\Phi}(t)$ are some numbers and $P_t$ are orthogonal projections onto a corresponding eigenspace $V_t$ of $L^2(X)$.
\end{thm}

\begin{remark}
Note that we will need the exact form of a limit only for a proof of Theorem \ref{pointwise}, otherwise we just need the fact that the limit exists. 
\end{remark}

Another important result of \cite{krashalom2025} is the following adaptation of van der Corput lemma for $\IP$ averages.

\begin{lemma}[van der Corput lemma for increasing IP-Følner sequences]\label{vdc}

Let $\mathcal{H}$ be a Hilbert space with norm $\| \cdot \|$ and inner product $\inner{\cdot}{\cdot}$, and let $(x_n)_{n\in\N}$ be a sequence in $\mathcal{H}$. For a sequence $(n_j)_{j\in\N}$ and an increasing F\o lner sequence $\Phi = (\Phi_N)_{N\in\N}$, we have
$$
\begin{aligned}
& \limsup_{N\to\infty}
\left\|
\E_{n\in \IP_{\Phi_N}((n_j)_{j\in\N})} x_n
\right\|^2
\leq \\ & \qquad
\limsup_{M\to\infty}
\E_{m_1,m_2\in \IP_{\Phi_M}((n_j)_{j\in\N})}
\limsup_{N\to\infty}
\E_{n\in \IP_{\Phi_N^M}((n_j)_{j\in\N})}
\inner{x_{n+m_1}}{x_{n+m_2}}.
\end{aligned}
$$
\end{lemma}
When $\ni$ diverges sufficiently fast, $\IP\bigl(\ni\bigr)$ is a set of density zero. Thus, it will be convenient to define a relative notion for positive density in $\IP$s.

\begin{dfn}\label{ipratden}
Let $S\subseteq \N$ be a set. We say that $S$ has \textit{positive lower $\IPr$ density} if the relative density of $S$ in every $\IP$ with rational spectrum is positive. Formally, if for every sequence $\ni$ with rational spectrum and every increasing F\o lner sequence $\Phi=(\Phi_N)_{N\in\mathbb{N}}$ we have 
$$
\liminf_{N\rightarrow\infty } \frac{|\IP_{\Phi_N}(\ni)\cap S|}{|\IP_{\Phi_N}(\ni)|} > 0.
$$
\end{dfn}

\section{Characteristic factors for polynomial ergodic averages along $\IP$ sets}\label{polycharchapter}

This section follows closely Leibman's argument from \cite{polycharfactor}. Our first goal is to show that the Host--Kra factors are characteristic for multiple ergodic averages  in nonlinear and rationally independent polynomial iterates along rational IPs (although, note that our final result is that the rational Kronecker factor is in fact the minimal characteristic factor).

In order to formulate our main result we will need some definitions. First, let us define \textit{polynomial systems}.
\begin{dfn}
Let $r \in \N$, called a \textit{length}. A \textit{polynomial system} is a finite tuple
$$
P = (p_1,\ldots,p_r)
$$
of distinct integer-valued polynomials $p_i:\Z \to \Z$. 
We define the degree of $P$ by
$$
\deg P := \max_{1 \leq i \leq r} \deg p_i.
$$
The system $P$ is called \textit{standard} if each $p_i$ is non-constant, $p_i - p_j$
is non-constant for all $1 \leq i < j \leq r$, and $\deg(p_1) = \deg P$.
\end{dfn}

We adapt the notion of characteristic factors to averages along a rational $\IP$.
\begin{dfn}[$\IPr$-characteristic factor]\label{ipratchar}
Let $P = (p_1,\dots,p_r)$ be a polynomial system and $\mps$ a measure preserving system. A factor $ (Y,\mathcal{C},\nu,S)$ of $X$ is called an \textit{$\IP_{rat}$-characteristic} factor with respect to $P$, if for any sequence $\ni$ with rational spectrum and any increasing F\o lner sequence $\Phi = \fol$ we have 
$$
\lim_{N\rightarrow\infty}\left \| \E_{\IP_{\Phi_N}(\ni)}\prod T^{p_i(n)}f_i - \E_{\IP_{\Phi_N}(\ni)}\prod T^{p_i(n)}E(f_i \mid Y)\right \|_{L^2(X)} =0
$$
for all $f_1,\dots ,f_r\in L^{\infty}(X)$. Equivalently, $$
\lim_{N\rightarrow\infty}\left \| \E_{\IP_{\Phi_N}(\ni)}\prod T^{p_i(n)}f_i\right \|_{L^2(X)}=0
$$ whenever there exists some $i\in\{1,...,r\}$ such that $E(f_i\mid Y)=0$.
\end{dfn} 

%The equivalent definition that we will mostly use is the following:
%\begin{dfn}
%Let $P = (p_1,\dots,p_r)$ be a polynomials sequence, $\mps$ be a measure preserving system and $(Y,\mathcal{C},\nu,S)$ be a factors of $X$. Then $Y$ is called an \textit{$\IP_{rat}$ characteristic} factor of $P$, if for any sequence $\ni$ with rational spectrum and any increasing F\o lner sequence $\Phi = \fol$ we have that 
%$$
%\left \| \E_{\IP_{\Phi_N}(\ni)}\prod T^{p_i(n)}f_i\right \|_{L^2(X)} \xrightarrow[]{N\rightarrow \infty}0
%$$
%for all $f_1,\dots ,f_r\in L^{\infty}(X)$ such that there is an $i\in \{1,\dots,r\}$ with $E(f_i \mid Y) = 0$.
%\end{dfn} 

\subsection{Characteristic factors for linear averages along a rational $\IP$}
The multilinear case (i.e., when all the polynomials involved in the system $P$ are linear) requires a separate treatment. Fortunately, this case was already established in \cite{krashalom2025}.

\begin{thm}\label{ipcharlinear}
For each polynomial system $P$ with $\deg P = 1$, there is a $k\ge 1$, such that $Z_{k-1}(X)$ is an $\IPr$-characteristic factor of $P$.
\end{thm}

\subsection{Preliminary lemma}

In this subsection we prove a technical lemma that we will need for an induction step in proof of the Theorem~\ref{polychar}.

\begin{dfn}[$\IP$-small sets]
Let $A$ be a subset of $\Z^d$, and denote by $\ind{A}:\Z^d\rightarrow \{0,1\}$ be the indicator of $A$. We say that $A$ is $\IP$-small if for each sequence $\ni$ and an increasing F\o lner sequence $\Phi = \fol$ we have that 
$$
\lim_{N\rightarrow \infty}\E_{n_1,\dots,n_d\in \IP_{\Phi_N}(\ni)}\ind{A}(n_1,\dots,n_d) = 0
$$
\end{dfn}
In order to apply the van der Corput lemma iteratively we will need the following lemma that guarantees that at each step the number of constant polynomials is $\IP$-small.

\begin{lemma}\label{ipsmalllemma}
Let $p,q:\Z\to\Z$ be polynomials, and assume that at least one of them has degree greater than $1$. Then the set
$$
S:=\{(v,w)\in \Z^2: p(x+v)-q(x+w)\text{ is constant in }x\}
$$
is $\IP$-small.
\end{lemma}
\begin{proof}
We first show that $S$ is contained in a line of the form
$$
\{(v,w)\in \Z^2: v-w=c\}
$$
for some $c\in \Z$. Let
$$
n:=\max(\deg p,\deg q).
$$
By assumption, either $p$ or $q$ has degree greater than $1$ and so $n>1$. Write
$$
p(x)=a_nx^n+a_{n-1}x^{n-1}+\cdots,
\qquad
q(x)=b_nx^n+b_{n-1}x^{n-1}+\cdots,
$$
where we allow $a_n=0$ if $\deg p<n$ and $b_n=0$ if $\deg q<n$.\\

We may assume that $a:=a_n=b_n\ne 0$, as otherwise $S= \varnothing$ and the lemma obviously holds. Since $p(x+v)-q(x+w)$ is constant in $x$, the coefficient of $x^{n-1}$ must vanish. Expanding the brackets, we see that this coefficient is equal to
$$
(na v+a_{n-1})-(na w+b_{n-1})
=
na(v-w)+a_{n-1}-b_{n-1}.
$$
Therefore
$$
na(v-w)+a_{n-1}-b_{n-1}=0,
$$
and hence
$$
v-w=\frac{b_{n-1}-a_{n-1}}{na}.
$$
If the number on the right hand side is not an integer, then no integer pair $(v,w)$ can satisfy this equation and $S=\emptyset$. Otherwise, writing
$$
c:=\frac{b_{n-1}-a_{n-1}}{na}\in\Z,
$$
we get
$$
S\subseteq H_c:=\{(v,w)\in\Z^2:v-w=c\}.
$$

It remains to show that $H_c$ is $\IP$-small. Let $(n_i)_{i\in\N}$ be a sequence of natural numbers, and let $\Phi=(\Phi_N)_{N\in\N}$ be a Følner sequence. We estimate
$$
\E_{v,w\in F_N} \ind{H_c}(v,w).
$$
Equivalently, we count pairs of subsets $(A,B)$ of $\Phi_N$ satisfying
$$
\sum_{i\in A}n_i-\sum_{i\in B}n_i=c.
$$
Fix $B\subseteq \Phi_N$. Then the admissible sets $A\subseteq \Phi_N$ satisfy
$$
\sum_{i\in A}n_i=\sum_{i\in B}n_i+c.
$$
The collection of such $A$ forms an antichain. Indeed, if $A\subsetneq A'$, then, since all $n_i$ are positive,
$$
\sum_{i\in A'}n_i>\sum_{i\in A}n_i,
$$
so $A$ and $A'$ cannot both have the same sum.

By Sperner's theorem, for each fixed $B$, the number of possible $A$ is at most
$$
\binom{|\Phi_N|}{\lfloor |\Phi_N|/2\rfloor}.
$$
Since there are $2^{|\Phi_N|}$ choices for $B$, we get
$$
\E_{v,w\in F_N} \ind{H_c}(v,w)
\le
\frac{
2^{|\Phi_N|}
\binom{|\Phi_N|}{\lfloor |\Phi_N|/2\rfloor}
}{
2^{2|\Phi_N|}
}
=
\frac{
\binom{|\Phi_N|}{\lfloor |\Phi_N|/2\rfloor}
}{
2^{|\Phi_N|}
}\xrightarrow{N\rightarrow \infty} 0,
$$
where the last limit follows from standard manipulations involving the Stirling approximation.
\end{proof}

\subsection{$\mathrm{PET}$ induction for standard polynomial families}
We will proceed by PET induction (first introduced by Bergelson in \cite{BergelsonPET}). First, we define a well ordering on polynomial families. 
\begin{dfn}
Let $P = (p_1,\dots,p_r)$ be a polynomial system. We say that two polynomials $p,q$ are \emph{equivalent} if $\deg p = \deg q$ and $\deg (p-q) < \deg p $.

For all $j = 1,\dots,\deg P$, let $\omega_j$ denote the number of equivalence classes of members of $P$ of degree $j$. The \textit{weight} of the family $P$ is defined to be $\omega(P) \colon= (\omega_1,\dots,\omega_{\deg P})$.

This defines a well-order on polynomial families by embedding the weights in $\mathbb{N}^\infty$ and ordering lexicographically. 
\end{dfn}

\begin{ex}
For $P = (n^3+n,n^2+2n,2n)$, then $\omega (P) = (1,1,1)$. If $Q = (n^4+n^2,n^4+n,3n,2n,n)$, then $\omega(Q) = (3,0,0,1)$. In this case, $\omega(P)< \omega(Q)$.
\end{ex}

The next proposition is an IP adaptation of a result of Leibman (see \cite[Proposition 9]{polycharfactor}). Our proof is almost verbatim, where we replace every instance of the van der Corput lemma with its IP-adaptation (i.e., Lemma \ref{vdc}). We provide a proof for the sake of completeness.
\begin{prop}[$Z_{k-1}(X)$ is $\IPr$-characteristic for standard polynomial systems]\label{polycharforstandart}
Let $\mps$ be an invertible measure preserving system and $\omega = (\omega_1,\dots,\omega_t)\in \mathbb{N}^t$ for some $t\in \N$. Then for all $r\geq 1$, there exists $k\ge 1$, such that for all integer polynomial families $p_1, \dots , p_r : \Z \rightarrow \Z$  , such that the system $P = (p_1,\dots,p_r)$ is standard with weight $\omega$, and for all $f_1,\dots, f_r\in L^\infty (X)$ with $E(f_1\mid Z_{k-1}(X)) = 0$, all rational $\ni$ and all increasing F\o lner sequences $\Phi = (\Phi_N)_{N\in \N}$, we have 

$$
\lim_{N\to\infty}\left \|\E_{ n\in F_N}\prod_{i=1}^r T^{p_i( n)}f_i\right \|^2_{L^2(X)} = 0.
$$

\begin{proof}
We prove the proposition by PET induction on the weight $\omega$.  The
linear case follows from the linear characteristic factor result for $\IP$-averages (see., \cite{krashalom2025}), so we may assume that $\deg P\geq 2$.

Fix $r\in \mathbb{N}$ and a weight $\omega$, and assume inductively that the
claim is known for every standard system of weight $\omega'<\omega$. Choose $k$ sufficiently large so that the claim holds for all such systems and
of length at most $2r$.\\
\noindent Let $
P=(p_1,\ldots,p_r)
$ be a standard system of weight $\omega$ and let $f_1,\ldots,f_r\in L^\infty(\mu)$ with
$$
E(f_1\mid Z_{k-1}(X))=0.
$$
By multilinearity and approximation, we may assume that
$$
\|f_i\|_\infty\leq 1
\qquad\text{for all } i=1,\ldots,r.
$$
Let $(n_j)_{j\in\mathbb{N}}$ be a sequence with  rational spectrum and let
$\Phi=(\Phi_N)_{N\in\mathbb{N}}$ be an increasing Følner sequence.  We write
$F_N=\IP_{\Phi_N}((n_j)_{j\in\mathbb{N}})$. Next, write 
$$
I_1:=\{i:\deg p_i=1\},
\qquad
I_2:=\{i:\deg p_i\geq 2\}.
$$
Since $P$ is standard and $\deg P\geq 2$, we have $1\in I_2$.

If $r=1$, put $i_0 = 1$. Otherwise, choose $i_0\in\{2,\ldots,r\}$ as follows: if not all
polynomials have the same degree, choose $p_{i_0}$ of minimal degree; if all
have the same degree but are not all equivalent to $p_1$, choose $p_{i_0}$
not equivalent to $p_1$; and if all are equivalent, choose any
$i_0\neq 1$.

For $v,w\in\mathbb{Z}$, define the auxiliary family
$$
P_{v,w}:=
\{p_i(n+v),p_i(n+w):i\in I_2\}
\cup
\{p_i(n+w):i\in I_1\}.
$$
Order this family as
$$
P_{v,w}=(q_{v,w,1},\ldots,q_{v,w,s})
$$
so that
$$
q_{v,w,1}(n)=p_1(n+v),
\qquad
q_{v,w,s}(n)=p_{i_0}(n+w).
$$
By Lemma \ref{ipsmalllemma}, the set of pairs $(v,w)$ for which
$P_{v,w}$ is not standard is $\IP$-small. Thus, we may define 
$$
P'_{v,w}:=
(q_{v,w,1}-q_{v,w,s},\ldots,q_{v,w,s-1}-q_{v,w,s}).
$$
For $\IP$-almost every $(v,w)$, this is a standard system.  Moreover, by the
standard PET reduction, the weight of $P'_{v,w}$ is strictly smaller than $\omega$:
$$
\omega(P'_{v,w})<\omega.
$$

The idea behind building $P'$ is that it is precisely the family that we will get after applying the $\IP$ van der Corput lemma (Lemma~\ref{vdc}) and using the $T$ invariance of $\mu$. 
Indeed,

$$
\begin{aligned}
& \limsup_{N\to\infty}
\left\|
\E_{ n\in F_N}\prod_{i=1}^r T^{p_i( n)}f_i
\right\|^2 \leq \\ & \qquad
\limsup_{M\to\infty}
\E_{ m, m'\in F_M}
\limsup_{N\to\infty}
\E_{ n\in (F_N^M)}
\int_X \prod_{i=1}^r T^{p_i( n +  m)}f_i \prod_{i=1}^r T^{p_i( n +  m')}\overline{f_i} \,d\mu \\  &
\qquad =
\limsup_{M\to\infty}
\E_{ m, m'\in F_M}
\limsup_{N\to\infty}
\E_{ n\in (F_N^M)}
\int_X \prod_{i\in I_2} T^{p_i( n +  m)}f_i \cdot T^{p_i( n +  m')}\overline{f_i}\prod_{i\in I_1} T^{p_i( n +  m')}\left(
\overline{f_i} T^{p_i(m)-p_i(m')}f_i \right ) \,d\mu \\ &
\qquad =
\limsup_{M\to\infty}
\E_{ m, m'\in F_M}
\limsup_{N\to\infty}
\E_{ n\in (F_N^M)}
\int_X h_{m,m',s} \prod_{i= 1}^{s-1} T^{(q_{m,m',i}-q_{m,m',s})(n)}h_{m,m',i}\,d\mu \le \\&
\qquad \le
\limsup_{M\to\infty}
\E_{ m, m'\in F_M}
\|h_{m,m',s}\|_{L^\infty}\limsup_{N\to\infty}
\left \|\E_{ n\in F_N^M}\prod_{i= 1}^{s-1} T^{(q_{m,m',i}-q_{m,m',s})(n)}h_{m,m',i} \right \|_{L^2(\mu)}
\end{aligned}
$$

Where each $h_{m,m',i}$ is equal to $f_i$, $\overline{f_i}$ or $\overline{f_i}T^{p_i(m)-p_i(m')}f_i$. Since $p_1$ is nonlinear (otherwise we would be in a base case), we have that $h_{m,m',1} = f_1$, so applying the induction hypothesis to a system $P'_{m,m'}$, we get that for  all $m,m'$  outside of an $\IP$-small set 
$$
\limsup_{N\to\infty}
\left \|\E_{ n\in F_N^M}\prod_{i= 1}^{s-1} T^{(q_{m,m',i}-q_{m,m',s})(n)}h_{m,m',i} \right \|_{L^2(\mu)} = 0
$$
And since the rest of the elements are bounded, we get that 
$$
\limsup_{M\to\infty}
\E_{ m, m'\in F_M}
\|h_{m,m',s}\|_{L^\infty}\limsup_{N\to\infty}
\left \|\E_{ n\in F_N^M}\prod_{i= 1}^{s-1} T^{(q_{m,m',i}-q_{m,m',s})(n)}h_{m,m',i} \right \|_{L^2(\mu)} = 0.
$$
This completes the proof.

\end{proof}
\end{prop}

\subsection{Reduction from general families to a standard case}The next proof follows closely the proof of Theorem 3 from \cite{polycharfactor}. 

\begin{thm}[The Host--Kra factors are $\IPr$-characteristic for nonlinear polynomials]\label{polychar}
For each polynomial system $P$ of nonlinear polynomials, there is a $k\ge 1$, such that $Z_{k-1}(X)$ is an $\IPr$-characteristic factor of $P$.
\end{thm}
\begin{proof}
Let $P=(p_1,\ldots,p_r)$ be a polynomial system. By changing polynomials up to constants, and then grouping the functions that are averaged by the same polynomial, we may assume that the
polynomials are nonconstant, have zero constant term, and are essentially
distinct, meaning that $p_i-p_j$ is nonconstant whenever $i\neq j$.

It is enough to prove that for every
$a\in\{1,\ldots,r\}$ there exists $k_a\geq 1$ such that, whenever
$$
E(f_a\mid Z_{k_a-1}(X))=0,
$$
we have
$$
\lim_{N\to\infty}
\left\|
\E_{n\in F_{N}}
\prod_{i=1}^r T^{p_i(n)}f_i
\right\|_{L^2(\mu)}
=0.
$$
Without loss of generality we may assume that $a = 1$.

Let
$$
b=\deg P=\max_{1\leq i\leq r}\deg p_i.
$$
Choose an integer-valued polynomial $q:\Z\to\Z$ of degree $b+1$ and
with $q(0)=0$, for example $q(n)=n^{b+1}$ will work.

Now, applying the $\IP$ van der Corput lemma \ref{vdc}, we get that 

$$
\begin{aligned}
&\limsup_{N\to\infty}\|\E_{n\in F_N} \prod_{i=1}^r T^{p_i(n)}f_i\|_{L^2(\mu)}^2
\leq \\ & \qquad
\limsup_{M\to\infty}
\E_{m,m'\in F_M}
\limsup_{N\rightarrow\infty}
\E_{n\in F_N^M}
\int_X
\prod_{i=1}^r T^{p_i(n+m)}f_i
\prod_{i=1}^r T^{p_i(n+m')}\overline{f_i}
\,d\mu \\ & \qquad
= \limsup_{M\to\infty}
\E_{m,m'\in F_M}
\limsup_{N\rightarrow\infty}\E_{n\in F_N^M}
\int_X
\prod_{i=1}^r T^{p_i(n+m)+q(n)}f_i
\prod_{i=1}^r T^{p_i(n+m')+q(n)}\overline{f_i}
\,d\mu \\ & \qquad
\le  \limsup_{M\to\infty} \E_{m,m'\in F_M}
\limsup_{N\rightarrow\infty}
\left \|
\E_{n\in F_N^M}
\prod_{i=1}^r T^{p_i(n+m)+q(n)}f_i
\prod_{i=1}^r T^{p_i(n+m')+q(n)}\overline{f_i}
\,
\right \|_{L^2(X)}
\end{aligned}
$$

Define a new polynomial family 
$$P'_{m,m'} = \{p_i(n+m)+q(n):i\in \{1,\dots,r\}\}\cup \{p_i(n+m')+q(n):i\in \{1,\dots,r\}\}$$
We order this family so that the first polynomial is
$
p_1(n+m)+q(n).
$
The family
$
P'_{m,m'}
$
is standard (note that this is the only place where we use the fact that each $p_i$ is nonlinear) for $\IP$-almost every $m,m'\in \Z$ (by Lemma~\ref{ipsmalllemma}). Now we may apply Proposition~\ref{polycharforstandart} to the system $P'_{m,m'}$, and thus complete the proof.
\end{proof}

\section{Pointwise convergence along $\IP$ sets over nilmanifolds}
In \cite{krashalom2025} Kra and the second author have established pointwise convergence for continuous functions on continuous toral systems  (see \cite[Theorem 5.3]{krashalom2025}). This implies a pointwise convergence on nilsystems $G/\Gamma$ for some dense family of functions $\mathcal{F}\subseteq L^2(\mu_{G/\Gamma})$. In this section we are strengthening this result by showing that $\mathcal{F}$ can be taken to be the family of continuous functions.

\subsection{Vertical characters}

We first need to establish some technical statement about a vertical characters on nilmanifolds.

\begin{dfn}
Let $X = G/\Gamma$ be a $k$-step nilmanifold. The function $f\in C(X)$ is called a \textit{vertical character} if there exists a $\Gamma_k$ invariant character $\chi :G_k\rightarrow S^1$, such that for all $g\in G_k$ and $x\in X$, we have $f(gx) = \chi(g)f(x)$. $\chi$ is called the \textit{frequency} of $f$.
\end{dfn}

\begin{prop}\label{vechardense}
Let $X = G/\Gamma$ be a $k$-step nilmanifold. The algebra generated by all vertical characters is dense in $C(X)$ with respect to the uniform norm.
\end{prop}
\begin{proof}
Let $A$ denote the algebra of all continuous vertical characters. By the Stone-Weierstrass theorem, it suffices to show that $A$ separates points. Let $x,y\in G/\Gamma$ be two distinct elements and write $x=g\Gamma,y=h\Gamma$ for some $g,h\in G$. Suppose first that $gG_k \Gamma\not=hG_k\Gamma$. In that case, there exists a continuous function $f\in C(G/G_k\Gamma)$ so that $f(gG_k \Gamma)\ne f(hG_k\Gamma)$. Lifting $f$ to $G/\Gamma$ we get a vertical character (with trivial frequency) satisfying that $f(g\Gamma)\not=f(h\Gamma)$. Otherwise, if $gG_k\Gamma=hG_k\Gamma$ then there exists some $u\in G_k$ so that $ux=y.$ Since $x$ and $y$ are distinct, $u\not\in \Gamma_k$ and so, from the Peter-Weyl theorem (the characters separate points), we can find a $\Gamma_k$-invariant character $\chi\in \widehat{G_k}$ so that $\chi(u)\neq 1$. Now, let $f\in C(X)$ be a sufficiently small bump function around point $h\Gamma$ and define a new function by setting $$
\phi'(x) \coloneqq \int_{G_k/\Gamma_k} f(tx)\overline{\chi(t)}\,dt.
$$ The continuity of $f$  and $\chi$ guarantees that $\phi'$ is also continuous, and $\phi'(h\Gamma) \neq 0$. Moreover, for any $v\in G_k/\Gamma_k$, we have $$\phi'(vx) = \int_{G_k/\Gamma_k} f(tvx)\overline{\chi(t)}\,dt = \int_{G_k/\Gamma_k} f(tx)\overline{\chi(tv^{-1})}\,dt = \chi(v)\phi'(x).$$
Thus, $\phi'\in A$ and $\phi'(h\Gamma) = \phi'(gu\Gamma)=\chi(u)\phi'(g\Gamma) \neq \phi'(g\Gamma)$, as required.
\end{proof}
\subsection{Formula for pointwise $\IP$ averages}

The Birkhoff pointwise ergodic theorem says that for an ergodic measure preserving system $\mps$ the ergodic averages of integrable $f$ converge to $\int_X fd\mu$ pointwise almost everywhere. Unfortunately, such a result when the averages are taken of a rational $\IP$ is not known. In fact, when the $\IP$ is not rational, the average may not converge at all, and even when it is rational the limit may not be  $\int_X f\,d\mu$. For example, if $\ni$ is a sequence of even numbers, and $X$ is a rotation on $\Z/2\Z$, then the $\IPr$-averages of any function $f$ will be equal to $f(0)$, whereas the integral is $\frac{f(0)+f(1)}{2}$. Our next goal is to show that when the system is an ergodic nilsystem, the distribution of the remainders of the elements of $\IP(\ni)$ modulo the number of connected components, and the integral of $f$ determine the limit. To do so, we introduce the \textit{remainder function} for some $r\in \N$.
\begin{dfn}
Let $r\in \N$, and let $i\in \Z$. Define $\mathbf{1}_{r,i} : \Z/r\Z\rightarrow \C$ by setting:
$$
\mathbf{1}_{r,i}(n)=
\begin{cases}
1, & n \equiv i \pmod{r},\\
0, & \text{otherwise}.
\end{cases}
$$
By an abuse of notation, we also write $\mathbf{1}_{r,i}(n)$ to denote $(\mathbf{1}_{r,i} \circ \pi_s)(n)$, where $\pi_s : \mathbb{Z}/r^s\Z \rightarrow \mathbb{Z}/r\mathbb{Z}$ is the natural projection for any $s\in \N$, same for a natural projection $\pi : \Z\rightarrow \Z/r\Z$.
\end{dfn}

\begin{dfn}
Let $\ni$ be a sequence with rational spectrum and $\Phi = (\Phi_N)_{N\in \N}$ an increasing F\o lner sequence. Define 
$$
\omega_{r,j}(\Phi,\ni) := \lim_{N\rightarrow\infty} \E_{n\in \IP_{\Phi_N}(\ni)} 1_{r,j}(n),
$$
and note that the limit exist by Theorem~\ref{ipmeanthm}, applied to $\Z/r\Z$ with $f = \mathbf{1}_{r,j}$.
\end{dfn}

Finally, we are ready to state and prove our pointwise result.
\begin{thm}[Pointwise limit formula along IPs for ergodic nilmanifolds]\label{pointwise}
Let $X = G/\Gamma$ be a $k$-step ergodic nilmanifold, let $\ni$ be a sequence with rational spectrum, and let $\Phi = (\Phi_N)_{N\in \N}$  be an increasing F\o lner sequence. Let $x\in X$ and let $Y_0,...,Y_{r-1}$ be the connected components of $X$, ordered such that $x\in Y_0$ and $T$ cyclically permutes them. Then for any $f\in C(X)$, we have
\begin{equation}\label{pointwiseeq}
\avgip T^nf(x) = \sum_{j=0}^{r-1}\omega_{r,j}(\Phi, \ni) \cdot\int_{Y_j}f \,d\mu_{Y_j},\end{equation}
where $\mu_{Y_j}$ is the Haar measure on $Y_j$.
\end{thm}

\begin{proof}
Note that since the algebra of vertical characters is dense in $C(X)$ (Proposition \ref{vechardense}), we may assume that $f$ is a vertical character. 

We proceed by inducting on the step $k$; first suppose that $k=1$. In that case, $T$ is a rotation on the compact abelian group $G/\Gamma$, and the vertical characters are eigenfunctions of $T$. In other words, $T^nf(x) = \alpha^nf(x)$, for some $\alpha\in \mathbb{T}$. 

If $\alpha$ is irrational, the integral of $f$ over each connected component is $0$. Indeed, in this case $f|_{Y_0}$ is a nontrivial eigenfunction of $T^r$ and is therefore orthogonal to $1$. On the other hand, from the assumption that $\ni$ has rational spectrum, the left hand side of \eqref{pointwiseeq} must also equal zero by definition.

If $\alpha$ is rational, then we argue that $f$ must be constant on each connected component. Indeed, we get that $f|_{Y_i}$ is $T^{rm}$ invariant for some $m\in \N$ for all $i$, so $f|_{Y_i}$ must be constant (a connected nilsystem is totally ergodic). But then $\int_{Y_i} f\,d\mu_{Y_i} = f|_{Y_i}$ and therefore \eqref{pointwiseeq} holds by definition. 

Let $k\geq 2$ and assume inductively that the theorem was established for all nilsystems of step less than $k$. Let $f$ be a vertical character. If the frequency of $f$ is trivial, then $f$ is measurable with respect to the $(k-1)$-step nilsystem $G/G_k\Gamma$. In that case the induction hypothesis applies and the claim holds. Otherwise, assume that the frequency $f$ is not trivial. In that case, the right hand side of \eqref{pointwiseeq} must be zero, because as Leibman showed in \cite[Lemma 2.10]{leibmanpointwise} (see also \cite[Proposition 2.2]{ackelsbergshalomrichter}), there are choices of $G$ and $\Gamma$ such that $G_i$ is connected for $i\ge 2$. Therefore, we have that for all $u\in G_k,$
$$
\int_{Y_j}f(x)\,d\mu_{Y_j} = \int_{Y_j}f(ux)\,d\mu_{Y_j} = \chi(u)\int_{Y_j}f(x)\,d\mu_{Y_j},
$$
and since $\chi$ is nontrivial we conclude that $\int_{Y_j}fd\mu =0$ for all $j = 0,\dots ,r-1$.\\ It is therefore left to show that the left hand side is also trivial. Choose any $x_0\in X$ and apply the $\IP$ van der Corput lemma (Lemma~\ref{vdc}) with the Hilbert space $\mathcal{H}=\mathbb{C}$ and $x_n = T^nf(x).$ We get 
\begin{equation}\label{pointwise2}
\limsup_{N\to\infty}
\left |
\E_{ n\in F_N}T^nf(x_0)
\right|^2
\leq
\limsup_{M\to\infty}
\E_{ m, m'\in F_M}
\limsup_{N\to\infty}
\E_{ n\in F_N^M}
 T^{n+m}f(x_0)T^{n+m'}\overline{f(x_0)}.
\end{equation}
Observe that $T^{m}f(x)T^{m'}\overline{f(x)}$ is $G_k$-invariant. Indeed, if $h\in G_k$, then $$T^{m}f(hx)T^{m'}\overline{f(hx)} = \chi(h)\overline{\chi(h)}T^{m}f(x)T^{m'}\overline{f(x)} = T^{m}f(x)T^{m'}\overline{f(x)}$$ (since $h$ commutes with $T$). Hence, $T^{m}f(x)T^{m'}\overline{f(x)}$ is measurable with respect to the $(k-1)$-step nilsystem $G/G_k\Gamma$. Applying the induction hypothesis with $T^mf\cdot T^{m'}\overline{f}$ at $x_0$, we deduce that right hand side of \eqref{pointwise2} equals to 
$$
\begin{aligned}
& \limsup_{M\to\infty}
\E_{ m, m'\in F_M} 
\sum_{j=0}^{r-1}\omega_{r,j}(\Phi,\ni)\int_{Y_j}
 T^{m}f(x)T^{m'}\overline{f(x)}d\mu_{Y_j} \\ &\qquad
= \limsup_{M\to\infty} 
\sum_{j=0}^{r-1}\omega_{r,j}(\Phi,\ni)\int_{Y_j}\left|\E_{ m\in F_M} T^{m}f\right|^2\,d\mu_{Y_j}
\\ &\qquad
\le
\limsup_{M\to\infty} r\int_X\left|\E_{ m\in F_M} T^{m}f\right|^2\,d\mu _{X} \\&\qquad =\limsup_{M\rightarrow\infty} r\left \|\E_{ m\in F_M} T^{m}f\right\|^2_{L^2(X)} = 0,
\end{aligned}
$$
where the inequality follows from the fact that $|\omega_{r,j}|\le 1$, and the last equality follows from Theorem~\ref{ipmeanthm} and the fact that $f$ has zero integral over each $Y_i$, $i = 0,\dots,r-1$, which means that it is orthogonal to $K_{\rat}(X)$.

\end{proof}

\section{$\IP$ distribution of polynomial sequences in nilmanifolds}
In \cite{leibmanpointwise} Leibman proved an equidistribution result for certain polynomial sequences in nilmanifolds. Our goal in this section is to derive a counterpart where the averages are taken along an $\IP$ with rational spectrum.  We begin with some preliminary definitions.
\begin{dfn}[Polynomial sequences]
Let $G=(G,\cdot)$ be a group. A \textit{polynomial sequence} in $G$ is a sequence of the form $g(n)= a_1^{p_1(n)}a_2^{p_2(n)}\cdots a_r^{p_r(n)}$ where $a_1,\dots,a_r\in G$, $p_1,\dots, p_r : \Z\rightarrow \Z$ are integer polynomials.
\end{dfn}
In the next definition we adapt the notion of a well-distributed sequence to the rational $\IP$ setting.
\begin{dfn}
Let $X = G/\Gamma$ be a nilmanifold and let $x\in X$. A sequence $\{g_nx\}_{n\in \N}$ of elements of $X$ is called \textit{$\IPr$-well-distributed} on $X$ if for any $f\in C(X)$, any sequence $\ni$ with rational spectrum and any increasing F\o lner sequence $\Phi = (\Phi_N)_{N\in \N}$ we have that
$$ 
\avgip f(g_nx) = \int_X f\,d\mu 
$$
\end{dfn}

\subsection{Equidistribution of an $\IP$ ergodic averages.}
Using a well known equidistribution result for nilmanifolds to extend the limit formula to non-ergodic nilsystems.
%We begin by extending the limit formula to nilsystems that are not necessarily ergodic.

\begin{prop}\label{linearfromleibman}
Let $X= G/\Gamma$ be a nilmanifold. For  every  $x\in X$ and $a \in G$ there exists a closed subgroup $E \le G$ such that $\overline{\{a^nx:n\in \mathbb{Z}\}} = Ex$ is a sub-nilmanifold of $X$. In addition, we have that for any $f\in C(X)$, any sequence $\ni$ with rational spectrum and any increasing F\o lner sequence $\Phi = \fol$,
$$
\avgip f(R_a^nx) = \sum_{j=0}^{r-1}\omega_{r,j}(\Phi,\ni)\int_{(Ex)_j}f\,d\mu_{(Ex)_j},
$$
where $(Ex)_j$, $j \in 0, \dots ,r-1$ are connected components of $Ex$, ordered such that $x\in (Ex)_0$ and $R_a$ cyclically permutes them.
\end{prop}

\begin{proof}
The existence of a closed $E\le G$ such that $\overline{\{a^nx:n\in \mathbb{Z}\}} = Ex$ was shown in \cite{leibmanpointwise}.

Extending $E$ if necessary we may assume that $a \in E$.

Fix $f \in C(X)$, a sequence $(n_i)_{i\in\N}$ with rational spectrum, and an increasing
Følner sequence $\Phi = (\Phi_N)_{N\in\N}$. The restriction of $f$ to $Ex$ is continuous, so we
may apply Theorem \ref{pointwise} to the ergodic nilsystem $(Ex,R_a)$. We deduce that
$$
\lim_{N\to\infty}
\E_{n\in \IP_{\Phi_N}((n_i)_{i\in\N})} f(R_a^n x)
=
\sum_{j=0}^{r-1}
\omega_{r,j}(\Phi,(n_i)_{i\in\N})
\int_{(Ex)_j} f \, d\mu_{(Ex)_j},
$$
as required.

\end{proof}

\subsection{Equdistribution of orbits of a unipotent automorphism}
Leibman (\cite{leibmanpointwise}) observed that every polynomial sequence on $G$ is associated with a unipotent automorphism (see Proposition \ref{polyred} below). Hence, we begin by proving the limit formula for unipotent automorphism (see Lemma \ref{unipotnetlemma}).

\begin{dfn}
An automorphism $\tau$ of $G$ is called \textit{unipotent} if the mapping $\xi : G \rightarrow G$ defined by $\xi (g) = \tau(g)g^{-1}$, satisfies $\xi^{\circ q} \equiv e_G$, for some $q \in \N$.
\end{dfn}

Given some nilmanifold $X= G/\Gamma$ and a unipotent automorphism $\tau$ the following construction allows us to think of $\tau$ as a translation by some element in $G$. 
\begin{dfn}
Let $X = G/\Gamma$ be a nilmanifold and let $\tau$ be a unipotent automorphism of $G$ such that $\tau(\Gamma)= \Gamma$. Define an \textit{extension} of $G$ by $\tau$ to be $\widehat{G} = G \rtimes \Z$, with product defined $(g,n)\cdot (h,m) = (g \tau^n(h),n+m)$. Identify $\tau$ with $(e_G,1)$ and $G$ with the subgroup $G\times \{0\}\leq \widehat{G}$ and let $\widehat{\Gamma} = \langle \Gamma,\tau \rangle$.
\end{dfn}

\begin{remark}
It is standard result (see \cite{leibmanpointwise}) that $\widehat{G}$ is a nilpotent Lie group and $\widehat{\Gamma}$ is its closed co-compact subgroup. In particular, $\widehat{G}/ \widehat{\Gamma}$ is a nilmanifold, which is isomorphic to $X$.
\end{remark}

We are set to prove the equidistribution result for unipotent automorphisms.
\begin{lemma}[Equidistribution for unipotent automorphisms]\label{unipotnetlemma}
Let $X=G/\Gamma$ be a nilmanifold, let $\tau$ a unipotent automorphism of $G$, and let $x\in X$. Then there is a connected closed subgroup $H \le G$ and points $x_0,...,x_{r-1}\in X$, such that $\overline{\{\tau^{n}x\}}_{n\in \Z} = \cup_jHx_j$, $x\in Hx_0$ and $\tau$ cyclically permutes $Hx_j$. Additionally, for every $f\in C(X)$, every sequence $\ni$ with rational spectrum, and every increasing F\o lner sequence $\Phi = \fol$, we have that 
$$
\avgip f(\tau^nx) = \sum_{j=0}^{r-1}\omega_{r,j}(\Phi,\ni)\int_{Hx_j}f\,d\mu_{Hx_j}.
$$
where $\mu_{Hx_j}$ is a Haar measure on $Hx_j$.
\end{lemma}
\begin{proof}

Apply Proposition \ref{linearfromleibman} to the nilrotation $L_{\tau}$ on $X = \widehat G/\widehat \Gamma$. We obtain a
closed subgroup $E\leq \widehat G$, which we may take to contain $\tau$, such that
$$
\overline{\{\tau^n x:n\in\Z\}}=E x,
$$
and $E x$ is a sub-nilmanifold of $ X$. Let $
H:=E^o
$ denote be the connected component of the identity in $E$. Since $\widehat{G}$ is an extension of $G$ by a discrete group we must have $$
H\leq G \leq \widehat G.
$$
We have that $Hx$ is a connected component of $Ex$,
and therefore the connected components $Y_0,\dots,Y_{r-1}$ of $Ex$ must be translates of $Hx$ by some $b_0,\dots, b_{r-1}\in E$. Since $H$ is normal we have $Y_j = b_jHx = Hb_jx$. In particular, we may take $x_j = b_jx$. Now, the limit formula can be obtained just by applying Theorem \ref{pointwise} to an ergodic nilsystem $(Ex,\tau)$.
\end{proof}

\subsection{Reduction to unipotent automorphisms.}
It \cite[Proposition 3.14]{leibmanpointwise}, Leibman showed that any polynomial sequence $g$ can be realized as unipotent transformation on an extension.

\begin{prop}[polynomial sequences as unipotent automorphisms]\label{polyred}
 Let $G$ be a nilpotent Lie group and let $g(n)$ be a polynomial sequence in $G$. There exists a nilpotent Lie group $\widetilde{G}$ with a discrete co-compact subgroup $\widetilde{\Gamma}$, an epimorphism $\eta \colon \widetilde{G} \to G$ with $\eta(\widetilde{\Gamma}) \subseteq \Gamma$, a unipotent automorphism $\tau$ of $\widetilde{G}$ with $\tau(\widetilde{\Gamma}) = \widetilde{\Gamma}$, and an element $c \in \widetilde{G}$ such that
$$
g(n)=\eta\bigl(\tau^n(c)\bigr), \quad \text{for all } n\in \Z.
$$
\end{prop}

Using this proposition we are now set to compute the limit formula for a polynomial sequences.

\begin{thm}[Limit formula for polynomial sequence]\label{bfromleibman}
Let $X=G/\Gamma$ be a nilmanifold, let $g(n)$ be a polynomial sequence in $G$ and let $x\in X$. Then there exists a closed connected subgroup $H\le G$, and points $x_0,...,x_{r-1}\in X$, not necessarily distinct, such that $x\in Hx_0$, $g(n)x$ cyclically visits $Hx_j$, $\overline{\{g(n)x\}}_{n\in \Z}= \bigcup_{i=0}^{r-1}Hx_i$, and for all $f\in C(X)$, all $\ni$ with rational spectrum and all increasing F\o lner sequences $\Phi = (\Phi_N)_{N\in \N}$, we have 
$$
\avgip f(g(n)x) = \sum_{j=0}^{r-1}\omega_{r,j}(\Phi,\ni)  \int_{Hx_j}f\,d\mu_{Hx_j},
$$
where $\mu_{Hx_j}$ is a Haar measure on $Hx_j$.
\end{thm}
\begin{proof}
By changing the polynomial sequence by a constant, we may assume that $x = \pi(e_G)$. Take $\widetilde{G}, \widetilde{\Gamma}$ and $c$ as in Proposition~\ref{polyred}. The epimorphism $\eta$ factors to $\eta :\widetilde{G}/\widetilde{\Gamma}\rightarrow X$. By Lemma \ref{unipotnetlemma} there is $\widetilde{H}\le \widetilde{G}$ and $\widetilde{x}_0,\dots, \widetilde{x}_{r-1}\in \widetilde{X}$ such that the corresponding conditions are satisfied. Take $H = \eta (\widetilde{H})$ and $x_j = \eta(\widetilde{x_j})$, $j=0,\dots,r-1$. We have that each $Hx_j$ is a connected sub-nilmanifold of $X$, and Haar measure on $Hx_j$ is the push-forward of the Haar measure on $\widetilde{H}\widetilde{x}_j$ by $\eta$. So

$$
\begin{aligned}
&\avgip f(g(n)x) = 
\avgip f(\eta(\tau^n(c))) =\\ & \qquad 
= \sum_{j=0}^{r-1}\omega_{r,j}(\Phi,\ni)\int_{\widetilde{H}\widetilde{x}_j}f\circ \eta \,d\mu_{\widetilde{H}\widetilde{x}_j} = \sum_{j=0}^{r-1}\omega_{r,j}(\Phi,\ni)\int_{Hx_j}f\,d\mu_{Hx_j},
\end{aligned}
$$
where $\mu_{\widetilde{H}\widetilde{x}_j}$ is the Haar measure on the nilmanifold $\widetilde{H}\widetilde{x}_j$.
\end{proof}

\begin{cor}\label{afromleibman}
Let $X = G/\Gamma$ be a nilmanifold and $g(n)$ a polynomial sequence in $G$, then for all $f\in C(X)$, all $\ni$ with rational spectrum and all increasing F\o lner sequences $\Phi = (\Phi_N)_{N\in \N}$,
$\avgip f(g(n)x)$
exists for all $x\in X$.
\end{cor}

In the case where the nilmanifold is connected we have a nice analogy of Theorem C from \cite{leibmanpointwise}.

\begin{dfn}
Let $X$ be a nilmanifold, we define its \textit{maximal factor torus} to be $Z = G/([G,G]\Gamma)$. Denote by $p$ the projection to $Z$.
\end{dfn}
\begin{thm}\label{cfromleibman}
Let $X=G/\Gamma$ be a connected nilmanifold, let $Z$ denote its maximal torus, and let $g(n)$ be a polynomial sequences. Then for all $x\in X$ the following are equivalent:
\begin{enumerate}[label=(\roman*)]
    \item The sequence $\{(g(n)x)\}_{n\in \Z}$ is dense in $X$.
    \item $\{(g(n)x)\}_{n\in \Z}$ is $\IPr$-well-distributed on $X$.
     \item The sequence $\{(g(n)p(x))\}_{n\in \Z}$ is dense in $Z$.
\end{enumerate}
\end{thm}
\begin{proof}
$(i)\Rightarrow (ii)$:
If the orbit is dense, then taking $H$ and $x_0,\dots,x_{r-1}$ as in Theorem \ref{bfromleibman}, we get that $\bigcup_{i=0}^{r-1} Hx_i = X$, and since $X$ is connected, we get that $Hx_0=\dots =Hx_{r-1} = X$. Hence, 
$$
\avgip f(g(n)x) = \sum_{j=0}^{r-1}\omega_{r,j}(\Phi,\ni)\cdot \int_{X}f\,d\mu = \int _X f\,d\mu.
$$ for all increasing F\o lner sequences $\Phi$, all $\ni$ with rational spectrum  and all $f\in C(X)$. Equivalently, $\{g(n)x\}_{n\in \Z}$ is $\IPr$-well-distributed. 

$(ii)\Rightarrow (i)$: Assume that the sequence is not dense, take the function $f$ that is a bump function supported outside of a closure of an orbit, then we will obviously get a contradiction.

$(i)\iff(iii)$: Was established by Leibman \cite[Theorem C]{leibmanpointwise}. 
\end{proof}

\section{Characteristic factors for rationally independent polynomials}

\subsection{Equidistribution of diagonal orbits}

In \cite{totergodic} Frantzikinakis and Kra showed that for a totally ergodic system the independent polynomial ergodic averages converge to a product of integrals. The following result about the distribution of polynomial orbits is particularly important for this paper.

\begin{thm}[polynomial equidistribution]\label{totergwd}
Let $(X,T)$ be a totally ergodic nilsystem and let $p_1,...,p_k : \Z \rightarrow \Z$ be a rationally independent sequence of integer polynomials. Then for almost all $x\in X$, the sequence
$$
\{(T^{p_1(n)}x,...,T^{p_k(n)}x)\}_{n\in \Z}
$$
is dense in $X^k$.
\end{thm}

\subsection{Rational Kronecker is characteristic}
In this section we are strengthening the Theorem~\ref{polychar} by showing that $K_{\rat}(X)$ is the $\IPr$-characteristic factor for a family of nonlinear rationally independent integer polynomials.

\begin{thm}[The rational Kronecker factor is $\IPr$-characteristic]\label{ratkronisfact}
Let $\mps$ be an invertible measure preserving system, let $k\ge 1$, and let $p_1,\dots,p_k :\Z \rightarrow \Z$ be a rationally independent family of nonlinear integer polynomials. Then, the rational Kronecker factor $K_{\rat}(X)$ is an $\IPr$-characteristic factor for $P = (p_1,\dots,p_k)$.
\end{thm}
\begin{proof}
It follows from the ergodic decomposition that we may assume that the system $X$ is ergodic. In Theorem~\ref{polychar} we showed that $Z_s(X)$ is $\IPr$-characteristic for some $s\geq 1$. Therefore, we may assume that $X$ is an inverse limit of ergodic nilsystems and by standard approximation arguments we may further assume that $X$ itself is an ergodic nilsystem, denote its action by $T = R_a$. Let $f_1,...,f_k\in L^\infty(X)$ and suppose that $E(f_i \mid \Krat(X)) = 0$ for some $i=1,...,k$ we need to show that
$$
\avgip \prod_{i=1}^k f_i(T^{p_i(n)}x) = 0
$$ in $L^2(\mu)$.
Without loss of generality we may assume that $f_i=f_1$. Every $L^2$-function can be approximated by continuous functions and therefore by Lebesgue dominated convergence theorem it suffices to assume that $f_1,...,f_k$ are continuous and show that $$
\avgip \prod_{i=1}^k f_i(T^{p_i(n)}x) = 0.
$$ for $\mu$-almost every $x\in X$. 

Let $F = f_1\otimes f_2\otimes\cdots\otimes f_k\in C(X^k)$, and let $g:\Z \to G^k$ the polynomial sequence $g(n) = (a^{p_1(n)},\dots ,a^{p_k(n)})$. From Theorem~\ref{bfromleibman}, we have that there exists some closed connected subgroup $H\le G^k$, and $x_0,\dots,x_{t-1}$, such that 
$$
\avgip \prod_{i=1}^k f_i(T^{p_i(n)}x) = \sum_{j=0}^{t-1}\omega_{t,j}(\Phi,\ni)\cdot \int_{Hx_j}F\,d\mu_{Hx_j}.
$$

We will show that $H = (G^o)^k$ for almost all $x\in X$.

Let $X_1 = X^o, X_2,\dots,X_r$ be the connected components of $X$, then $(X_i,T^r)$ are totally ergodic nilsystems for all $i = 1,\dots, r$. Observe that $p_i(r^sn) = m_i + rq_i(n)$ for some sufficiently large $s \in \N$, some constant $m_i$ and an integer valued  polynomial $q_i$ of the same degree as $p_i$. Since $p_1,...,p_k$ are independent over $\Q$ so are $q_1,...,q_k$. Let $x\in X_i$ and apply Theorem~\ref{totergwd} with the system $(q_1,...,q_k)$ and $T^r$. We conclude that $(T^{rq_1(n)} x,...,T^{rq_k(n)}x)_{n\in \N}$ is dense in $X_i^k$ for almost every $x\in X_i$. Therefore,  $$
(T^{p_1(r^sn)}x,\dots,T^{p_k(r^sn)}x)_{n\in \N}
= (T^{m_1}(T^{r})^{q_1(n)}x,\dots,T^{m_k}(T^r)^{q_k(n)}x)_{n\in \N}$$
is dense in $X_{i_1}\times X_{i_2}\times\dots \times X_{i_k}$ for almost all $x\in X_i$, where $i_ j = i + m_j \pmod r$. Since
$\overline{(T^{p_1(r^sn)}x,\dots,T^{p_k(r^sn)}x)_{n\in \N}}\subseteq \overline{(T^{p_1(n)}x,\dots,T^{p_k(n)}x)_{n\in \N}} = \bigcup_{j=1}^r Hx_j$ (where both $H$ and $x_1,\dots,x_r$ depend on $x$). We see that $Hx_j$ is of positive measure in $X^k$ for some $j$. Since $\Gamma$ is discrete, so is the stabilizer of $x_j$ and therefore we must have that $H$ is of positive measure in $G^k$. Moreover, $H$ is closed and connected and therefore $H= (G^o)^k$.

If $H = (G^o)^k$, we have that 
$$
\int_{Hx_j}F\,d\mu_{Hx_j} = \left (\int_{X_{x_j,1}} f_1\,d\mu_{X_{x_j,1}}\right)\cdot \dots \left ( \int_{X_{x_j,k}}f_k\,d\mu_{X_{x_j,k}}\right ),
$$
where $X_{x_j,1},\dots,X_{x_j,k}$ are connected components of $X$. However, from the assumption $E(f_1 \mid K_{\rat}(X)) = 0$, it follows that $\int_{Y}f_1\,d\mu_Y = 0$ whenever $Y$ is a connected component. Thus,
$$
\avgip \prod_{i=1}^k f_i(T^{p_i(n)}x) = 0
$$
for almost all $x\in X$, as required. 
\end{proof}

\section{Applications}
\subsection{Ergodic large intersection property}

We use Theorem~\ref{ratkronisfact} to prove the $\IPr$ large intersection property (Theorem~\ref{main:thm}). In \cite{indepoly} it was shown that the set of large intersection is syndetic. The conceptual difference is that when averaging over $\N$ one may always write $\N$ as the union of $r\N + i$ for $i=0,\dots,r-1$, and thus pass to an average over the action $T^r$. However, unfortunately, in $\IP(\ni)$ the set of all elements with a given residue modulo $r$ is no longer an $\IP$. To overcome this technicality, we consider the following functions:
\begin{dfn}
Let $r, k\ge 1$. Define
\begin{align*}
    \psi_r(n_1,\dots,n_k) &= \mathbf{1}_{r,0}(n_1)\cdot \mathbf{1}_{r,0}(n_2)\cdot ... \cdot \mathbf{1}_{r,0}(n_k).
\end{align*} 
Moreover, for a polynomial system $P = (p_1,\dots,p_k)$, of integer valued polynomials define a function $\psi^P_r: \Z/r^s\Z \rightarrow \C$ as $$\psi^P_r(n) = \psi_{r}(p_1(n),p_2(n),\dots,p_k(n))$$
for $s$ large enough so that this function will be well defined (with $\psi_r$ and $\psi_r^P$ we do the same abuse of notation as with $\mathbf{1}_{r,j}$).
\end{dfn}

First we need to show that there are sufficiently many $n\in \IP(\ni)$ such that $\psi^P_r(n) = 1$. 

\begin{lemma}
Let $r,k \ge 1$, let $P = (p_1,\dots,p_k)$ be a polynomial system such that $p_i(0)=0$ for $i=1,\dots,k$, then  for any sequence $\ni$ with rational spectrum and any increasing F\o lner sequence $\Phi =(\Phi_N)_{N\in \N}$, we have
\begin{equation}\label{limitpsy}
\avgip \psi^P_r(n) >0.
\end{equation}
\end{lemma}

\begin{proof}
Define $X = G = (\Z/r^s\Z)^k$ for $s$ large enough and polynomial sequence in $G$ as $g(n) = (p_1(n),p_2(n),\dots, p_k(n))$, we have that $\psi_r\in C(X)$, and that $\psi_r(g(n)e) = \psi_r^P(n)$, therefore, by Corollary \ref{afromleibman}, we have that the limit exists.

Since we have that $\psi_r^P(n) \ge \ind{r^s,0}(n)$ for all $n\in \Z$ it is enough to show that 
$$
c := \avgip \ind{r,0}(n) > 0.
$$
for all $r\in \N$.

Assume without loss of generality that $\Phi_N$ are of the form $[1,a_N]$. First, partition $\N$ into a disjoint union of sets $A_0\cup A_1 \cup \cdots A_{r-1}$, such that for every $k \in A_i$, we have that $n_k \equiv i \pmod{r}$. Note that since $\mathbf{1}_{r,0}$ is non-negative, if some $A_m$ is finite, then it is enough to prove the statement for a sequence obtained after removing all elements with an index in $A_m$ from $\ni$.  Thus, we can assume that all $A_m$ are infinite or empty. Let $r_1,...,r_t$ be the numbers such that $A_{r_i}$ is infinite and enumerate them, 
$$(l_{r_1,i})_{i\in\N},\dots (l_{r_t,i})_{i\in \N}.$$

Consider the set $ T = \{n_{l_{r_i,j}} : 1\leq i \leq t, \space 1\leq j\leq r \}$ of the first $r$ elements in each set. We have that for any finite sum of elements from $\ni \backslash T$,  some elements of $T$ may be added to it so that the resulting sum will be divisible by $r$.

Since $T$ is finite, there exists $N$ sufficiently large so that $T\subseteq \Phi_N$. Thus, for all larger values of $N$ we see that at least $2^{|\Phi_N|-|T|}$ elements in $\IP_N(\ni)$ are divisible by $r$. In particular, we must have that $c\geq \frac{1}{2^{|T|}}$, as required.
\end{proof}

Another statement that we need is that the rational Kronecker factor remains characteristic when polynomials are multiplied for some general family of functions of $n$ (we will use it only for $\psi = \psi_r$, but the proof is the same so we state it more generally).

\begin{prop}\label{ratkroncharpsi}
Let $\mps$ be an invertible system, let $r\in \mathbb{N}$ and $P = (p_1,\dots,p_k)$ be polynomial system of nonlinear rationally independent polynomials, and let $f_1,\dots,f_k\in L^\infty(X)$. Then for all $\psi : (\Z/r\Z)^k\rightarrow \C $, all $\ni$ with rational spectrum and an increasing F\o lner sequence $\Phi = \fol$ we have that 
$$
\begin{aligned}
\lim_{N\rightarrow \infty}\left \| \E_{n\in F_N}\psi(p_1(n),\dots,p_k(n))\prod_{i=1}^k T^{p_i(n)}f_i - \E_{n\in F_N}\psi(p_1(n),\dots,p_k(n)) \prod_{i=1}^kT^{p_i(n)}\widetilde{f_i}\right \|_{L^2(X)}^2 = 0
\end{aligned}
$$
in $L^2(X)$, where $\widetilde{f_i} = E(f_i \mid K_{rat}(X))$.
\end{prop}
\begin{proof}
Using Fourier expansion of the function $\psi(n)$, it is enough to prove the statement with $\psi(n_1,\dots, n_k) = \prod_{j=1}^ke^{2\pi in_jq_j}$ for some $q_1,\dots,q_k\in \Q$ with denominator dividing $r$.

Let us define a new system $\overline{X} = X\times \Z/r\Z$, and define the action $\overline{T}(x,m) = (Tx,m+1)$ for $x\in X, m \in \Z/r\Z$. This system is invertible, $X$ is its factor, and $K_{rat}(\overline{X})=K_{rat}(X) \times \Z/r\Z$, so we may think of our functions $f_i$ and $\widetilde{f_i}$ as defined on $\overline{X}$. We have that

$$
\begin{aligned}
& \E_{n\in F_N}\left(\prod_{j=1}^ke^{2\pi i p_j(n)q_j}\right) \prod_{i=1}^k T^{p_i(n)}f_i - \E_{n\in F_N}\left(\prod_{j=1}^ke^{2\pi i p_j(n)q_j}\right)\prod_{i=1}^k T^{p_i(n)}\widetilde{f_i} =  \\ & \qquad
\left (\E_{n\in F_N}\prod_{j=1}^k \overline{T}^{p_j(n)}(\phi_j(x,m)f_j(x)) - \E_{n\in F_N}\prod_{j=1}^k\overline{T}^{p_j(n)}(\phi_j(x,m)\widetilde{f_j}(x))\right )\prod_{j=1}^k\overline{\phi_j(x,m)},
\end{aligned}
$$
where $\phi_j(x,m) = e^{2\pi i q_jm}$ is an eigenfunction with eigenvalue $e^{2\pi i q_j}$. Since $\phi_j$ is measurable with respect to $ K_{rat}(\overline{X})$, the limit must be zero from Theorem \ref{ratkronisfact} applied to a system $\overline{X}$ and functions $g_j(x,m) = \phi_j(x,m)f_j(x)$, noting that $\E(g_j \mid \Krat(\overline{X})) = \phi_j \cdot \E(f_j\mid \Krat(X))$.

\end{proof}

We are now ready to prove the main multiple recurrence result of our paper (Theorem~\ref{main:thm}):

\begin{thm*}
Let $\mps$ be an invertible measure preserving system, $k \ge 1$, and let $p_1,\dots,p_k : \Z \rightarrow \Z$ be a family of rationally independent polynomials. Assume further that $p_i(0) = 0$, $i = 1,\dots,k$. Then for any $A\in \mathcal{B}$ with $\mu(A) >0$, and all $\varepsilon > 0$, the set 
\begin{equation}
\left\{n\in \N: \mu(A\cap T^{p_1(n)}A\cap T^{p_2(n)}A\cap\dots\cap T^{p_k(n)}A) > \mu(A)^{k+1} - \varepsilon\right\} 
\end{equation}
has positive lower $\IPr$ density. 
\end{thm*}
\begin{proof}
Let $S_\varepsilon=\left\{n\in \N: \mu(A\cap T^{p_1(n)}A\cap T^{p_2(n)}A\cap\dots\cap T^{p_k(n)}A) > \mu(A)^{k+1} - \varepsilon\right\}$. Suppose, by contradiction, that $S_\epsilon$ has zero lower $\IPr$ density, then there exists a sequence $\ni$ with rational spectrum, an increasing F\o lner sequence $\Phi=\fol$ (note that we may need to take a subsequence, but it is still an increasing F\o lner sequence), such that \begin{equation}\label{Ssmall}\lim_{N\rightarrow\infty} \frac{|S_\varepsilon \cap \IP_{\Phi_N}(\ni)|}{|\IP_{\Phi_N}(\ni)|}=0.
\end{equation}
In particular, for all $n\not\in S_\varepsilon$, we have $$\mu(A\cap T^{p_1(n)}A\cap T^{p_2(n)}A\cap\dots\cap T^{p_k(n)}A) \leq \mu(A)^{k+1} - \varepsilon.
$$
Let $q(n)=n^b$ for some $b$ sufficiently large (e.g., $b=\deg(P)+1$) so that $Q=(q(n),p_1(n)+q(n),\dots,p_k(n)+q(n))$ is a polynomial system of rationally independent non-linear polynomials. Since $T$ is measure-preserving we have
\begin{equation}\label{contradiction}\mu(T^{q(n)}A\cap T^{q(n)+p_1(n)}A\cap T^{q(n)+p_2(n)}A\cap\dots\cap T^{q(n)+p_k(n)}A) \leq \mu(A)^{k+1} - \varepsilon
\end{equation}
for all $n\not\in S_\varepsilon.$
Denote $f = E(\ind{A} \mid K_{\rat}(X))$. It is a well known fact that the rational Kronecker factor can be approximated by finite systems, so take $r > 1$ such that for $f_r  = E(f\mid K_r(X))$, 
\begin{equation}\label{rapprox}\|f_r - f\|_{L^2(X)}< \frac{\varepsilon}{2(k+1)}.
\end{equation}

Put
$$
c = \avgip \psi^Q_r(n).
$$

Decomposing $F_N=(F_N\cap S_\varepsilon) \bigcup (F_N\cap S_\varepsilon^C)$ we see from \eqref{contradiction} and Proposition \ref{ratkroncharpsi} that,

\begin{align*}\lim_{N\rightarrow \infty}\E_{n\in F_N} \frac{1}{c}\psi_r^Q(n)\cdot \int_X T^{-q(n)}f\cdot\prod_{i=1}^k T^{-p_i(n)-q(n)}f\,d\mu&=\lim_{N\rightarrow \infty}\E_{n\in F_N} \frac{1}{c}\psi_r^Q(n)\cdot \mu(T^{q(n)}A\cap T^{p_1(n)+q(n)}A\cap...\cap T^{p_k(n)+q(n)}A)\\&\leq \lim_{N\rightarrow \infty }\left (\frac{|S_\varepsilon\cap F_N|}{c|F_N|} +\frac{1}{|F_N|}\sum_{n\in S_\varepsilon^C\cap F_N} \frac{1}{c}\psi^Q_r(n)\cdot (\mu(A)^{k+1}-\varepsilon) \right ). 
\end{align*}
In particular, by \eqref{Ssmall} we have
\begin{equation}\label{maineq1}
\lim_{N\rightarrow \infty}\E_{n\in F_N} \frac{1}{c}\psi_r^Q(n)\cdot \int_X T^{-q(n)}f\cdot\prod_{i=1}^k T^{-p_i(n)-q(n)}f\,d\mu  \leq  \mu(A)^{k+1}-\varepsilon.
\end{equation}

\noindent Applying the Cauchy-Schwarz inequality and \eqref{rapprox} iteratively, we see that for all  $a_1,...,a_k\in \Z$ we have:
$$
\left | \int_X f\cdot T^{a_1} {f} \cdot T^{a_2} {f} \dots T^{a_k} {f}\,d\mu - \int_X f_r\cdot T^{a_1} f_r \cdot T^{a_2} f_r \dots T^{a_k} f_r\,d\mu\right | < \varepsilon/2.
$$
Therefore,
$$
\begin{aligned}
&\limsup _{N\rightarrow \infty} \left | \E_{n\in F_N}\int _{X}\frac{\psi^Q_r(n)}{c}T^{-q(n)}f\cdot T^{-p_1(n)-q(n)}{f}\cdot ... \cdot T^{-p_k(n)-q(n)}{f} \,d\mu\right | >\\ & \qquad \limsup _{N\rightarrow \infty} \left | \E_{n\in F_N}\int _{X}\frac{\psi^Q_r(n)}{c}T^{-q(n)}f_r\cdot T^{-p_1(n)-q(n)}f_r\cdot ... \cdot T^{-p_k(n)-q(n)}f_r \,d\mu\right |-\varepsilon.
\end{aligned}
$$
On the other hand, since $f_r$ is measurable with respect to $K_r(X)$ we get that if $r \mid a$, then $T^af_r = f_r$. Note that this happens for $q(n)$ and $p_i(n)+q(n)$ for all $i$, whenever $\psi_r^Q(n)$ is nonzero, so we have 
$$
\begin{aligned}
&\limsup _{N\rightarrow \infty} \left | \E_{n\in F_N}\int _{X}\frac{\psi^Q_r(n)}{c}T^{-q(n)}f_r\cdot T^{-p_1(n)-q(n)}f_r\cdot ... \cdot T^{-p_k(n)-q(n)}f_r \,d\mu\right | = \\ & \qquad
\limsup _{N\rightarrow \infty} \left | \E_{n\in F_N}\frac{\psi^Q_r(n)}{c}\int _{X} f_r^{k+1}\,d\mu\right | \geq \mu(A)^{k+1}
\end{aligned}
$$
where the last inequality follows from Jensen's inequality. We conclude $$
\begin{aligned}
&\limsup _{N\rightarrow \infty} \left | \E_{n\in F_N}\int _{X}\frac{\psi^Q_r(n)}{c}T^{-q(n)}f\cdot T^{-p_1(n)-q(n)}{f}\cdot ... \cdot T^{-p_k(n)-q(n)}{f} \,d\mu\right | > \mu(A)^{k+1}-\varepsilon,
\end{aligned}
$$
contradicting \eqref{maineq1}.
\end{proof}

\printbibliography

@misc{krashalom2025,
      title={Ergodic averages and the large intersection property along IP sets}, 
      author={Bryna Kra and Or Shalom},
      year={2025},
      eprint={2506.17771},
      archivePrefix={arXiv},
      primaryClass={math.DS},
      url={https://arxiv.org/abs/2506.17771}, 
}

@article{bergelson2014rigidity,
  title={Rigidity and non-recurrence along sequences},
  author={Bergelson, Vitali and del Junco, Andreas and Lema{\'n}czyk, M and Rosenblatt, Joseph},
  journal={Ergodic Theory Dynam. Systems},
  volume={34},
  number={5},
  pages={1464--1502},
  year={2014},
  publisher={Cambridge University Press}
}

@article {moreexamples,
    AUTHOR = {Donoso, Sebasti\'an and Le, Anh Ngoc and Moreira, Joel and
              Sun, Wenbo},
     TITLE = {Optimal lower bounds for multiple recurrence},
   JOURNAL = {Ergodic Theory Dynam. Systems},
  FJOURNAL = {Ergodic Theory and Dynamical Systems},
    VOLUME = {41},
      YEAR = {2021},
    NUMBER = {2},
     PAGES = {379--407},
      ISSN = {0143-3857,1469-4417},
   MRCLASS = {37A44 (11K36 37A30)},
  MRNUMBER = {4177289},
MRREVIEWER = {Song\ Shao},
       DOI = {10.1017/etds.2019.72},
       URL = {https://doi.org/10.1017/etds.2019.72},
}

@misc{shalomvectorspace,
      title={On the Furstenberg-Katznelson constant for the IP Szemeredi theorem over finite fields}, 
      author={Or Shalom},
      year={2026},
      eprint={2604.05768},
      archivePrefix={arXiv},
      primaryClass={math.DS},
      url={https://arxiv.org/abs/2604.05768}, 
}

@article {underrecurrent,
    AUTHOR = {Boshernitzan, Michael and Frantzikinakis, Nikos and Wierdl,
              M\'at\'e},
     TITLE = {Under-recurrence in the {K}hintchine recurrence theorem},
   JOURNAL = {Israel J. Math.},
  FJOURNAL = {Israel Journal of Mathematics},
    VOLUME = {222},
      YEAR = {2017},
    NUMBER = {2},
     PAGES = {815--840},
      ISSN = {0021-2172,1565-8511},
   MRCLASS = {37A25 (37B20)},
  MRNUMBER = {3722267},
MRREVIEWER = {Song\ Shao},
       DOI = {10.1007/s11856-017-1606-8},
       URL = {https://doi.org/10.1007/s11856-017-1606-8},
}

@article {Hardyexamples,
    AUTHOR = {Bergelson, Vitaly and Moreira, Joel and Richter, Florian K.},
     TITLE = {Multiple ergodic averages along functions from a {H}ardy
              field: convergence, recurrence and combinatorial applications},
   JOURNAL = {Adv. Math.},
  FJOURNAL = {Advances in Mathematics},
    VOLUME = {443},
      YEAR = {2024},
     PAGES = {Paper No. 109597, 50},
      ISSN = {0001-8708,1090-2082},
   MRCLASS = {37A30 (05D10 11B30 28D05 37A44)},
  MRNUMBER = {4715172},
MRREVIEWER = {Thomas\ Ward},
       DOI = {10.1016/j.aim.2024.109597},
       URL = {https://doi.org/10.1016/j.aim.2024.109597},
}

@incollection {bergelsonultra,
    AUTHOR = {Bergelson, Vitaly},
     TITLE = {Ultrafilters, {IP} sets, dynamics, and combinatorial number
              theory},
 BOOKTITLE = {Ultrafilters across mathematics},
    SERIES = {Contemp. Math.},
    VOLUME = {530},
     PAGES = {23--47},
 PUBLISHER = {Amer. Math. Soc., Providence, RI},
      YEAR = {2010},
      ISBN = {978-0-8218-4833-3},
   MRCLASS = {05D10 (11B75 37A45)},
  MRNUMBER = {2757532},
MRREVIEWER = {Randall\ McCutcheon},
       DOI = {10.1090/conm/530/10439},
       URL = {https://doi.org/10.1090/conm/530/10439},
}

@incollection {bergelsonupdate,
    AUTHOR = {Bergelson, Vitaly},
     TITLE = {Ergodic {R}amsey theory---an update},
 BOOKTITLE = {Ergodic theory of {${\bf Z}^d$} actions ({W}arwick,
              1993--1994)},
    SERIES = {London Math. Soc. Lecture Note Ser.},
    VOLUME = {228},
     PAGES = {1--61},
 PUBLISHER = {Cambridge Univ. Press, Cambridge},
      YEAR = {1996},
      ISBN = {0-521-57688-1},
   MRCLASS = {28D05 (05A18 05D10 11B25)},
  MRNUMBER = {1411215},
MRREVIEWER = {Karl\ Petersen},
       DOI = {10.1017/CBO9780511662812.002},
       URL = {https://doi.org/10.1017/CBO9780511662812.002},
}

@article {BKM,
    AUTHOR = {Bergelson, Vitaly and H\aa land Knutson, Inger J. and
              McCutcheon, Randall},
     TITLE = {I{P}-systems, generalized polynomials and recurrence},
   JOURNAL = {Ergodic Theory Dynam. Systems},
  FJOURNAL = {Ergodic Theory and Dynamical Systems},
    VOLUME = {26},
      YEAR = {2006},
    NUMBER = {4},
     PAGES = {999--1019},
      ISSN = {0143-3857,1469-4417},
   MRCLASS = {37A15 (11B75 28D15)},
  MRNUMBER = {2246589},
MRREVIEWER = {Bryna\ Kra},
       DOI = {10.1017/S0143385706000010},
       URL = {https://doi.org/10.1017/S0143385706000010},
}

@article {BFM,
    AUTHOR = {Bergelson, Vitaly and Furstenberg, Hillel and McCutcheon,
              Randall},
     TITLE = {I{P}-sets and polynomial recurrence},
   JOURNAL = {Ergodic Theory Dynam. Systems},
  FJOURNAL = {Ergodic Theory and Dynamical Systems},
    VOLUME = {16},
      YEAR = {1996},
    NUMBER = {5},
     PAGES = {963--974},
      ISSN = {0143-3857,1469-4417},
   MRCLASS = {28D15 (05D10)},
  MRNUMBER = {1417769},
MRREVIEWER = {Thomas\ Ward},
       DOI = {10.1017/S0143385700010130},
       URL = {https://doi.org/10.1017/S0143385700010130},
}

@InCollection{bergelson2003minimal,
  Title                    = {{Minimal idempotents and ergodic Ramsey theory}},
  Author                   = {Bergelson, V.},
  Booktitle                = {Topics in Dynamics and Ergodic Theory, London Math.~Soc.~Lecture Note Ser.},
  Publisher                = {Cambridge Univ. Press},
  Year                     = {2003},
  Volume                   = {310},
}

@Article{furstenberg1977ergodic,
  Title                    = {{Ergodic behaviour of diagonal measures and a theorem of Szemer{\'e}di on arithmetic progressions}},
  Author                   = {H.~Furstenberg},
  Journal                  = {J. Anal. Math.},
  Year                     = {1977},
  Pages                    = {204-256},
  Volume                   = {31},
}

@article {BergelsonPET,
    AUTHOR = {Bergelson, V.},
     TITLE = {Weakly mixing {PET}},
   JOURNAL = {Ergodic Theory Dynam. Systems},
  FJOURNAL = {Ergodic Theory and Dynamical Systems},
    VOLUME = {7},
      YEAR = {1987},
    NUMBER = {3},
     PAGES = {337--349},
      ISSN = {0143-3857,1469-4417},
   MRCLASS = {28D05 (47A35)},
  MRNUMBER = {912373},
MRREVIEWER = {Steven\ Alpern},
       DOI = {10.1017/S0143385700004090},
       URL = {https://doi.org/10.1017/S0143385700004090},
}

@article {ABB,
    AUTHOR = {Ackelsberg, Ethan and Bergelson, Vitaly and Best, Andrew},
     TITLE = {Multiple recurrence and large intersections for abelian group
              actions},
   JOURNAL = {Discrete Anal.},
  FJOURNAL = {Discrete Analysis},
      YEAR = {2021},
     PAGES = {Paper No. 18, 91},
      ISSN = {2397-3129},
       DOI = {10.19086/da},
       URL = {https://doi.org/10.19086/da},
}

@article{shalom2,
AUTHOR = {Shalom, Or}, 
TITLE= {Multiple ergodic averages in abelian groups and Khintchine type recurrence},
Journal = {Trans. Amer. Math. Soc. 375 (2022), 2729-2761}, 
Year = 2021
}

@article {ABS,
    AUTHOR = {Ackelsberg, Ethan and Bergelson, Vitaly and Shalom, Or},
     TITLE = {Khintchine-type recurrence for 3-point configurations},
   JOURNAL = {Forum Math. Sigma},
  FJOURNAL = {Forum of Mathematics. Sigma},
    VOLUME = {10},
      YEAR = {2022},
     PAGES = {Paper No. e107, 57},
      ISSN = {2050-5094},
   MRCLASS = {37A15 (05D10 37A30)},
  MRNUMBER = {4519061},
       DOI = {10.1017/fms.2022.97},
       URL = {https://doi.org/10.1017/fms.2022.97},
}

@article {FranKuca,
    AUTHOR = {Frantzikinakis, Nikos and Kuca, Borys},
     TITLE = {Joint ergodicity for commuting transformations and
              applications to polynomial sequences},
   JOURNAL = {Invent. Math.},
  FJOURNAL = {Inventiones Mathematicae},
    VOLUME = {239},
      YEAR = {2025},
    NUMBER = {2},
     PAGES = {621--706},
      ISSN = {0020-9910,1432-1297},
   MRCLASS = {37A44 (05D10 11B30 28D05)},
  MRNUMBER = {4850605},
MRREVIEWER = {Ryo\ Moore},
       DOI = {10.1007/s00222-024-01313-w},
       URL = {https://doi.org/10.1007/s00222-024-01313-w},
}

@misc{jerg,
      title={Resolving the joint ergodicity problem for Hardy sequences}, 
      author={Sebastián Donoso and Andreas Koutsogiannis and Borys Kuca and Wenbo Sun and Konstantinos Tsinas},
      year={2025},
      eprint={2506.20459},
      archivePrefix={arXiv},
      primaryClass={math.DS},
      url={https://arxiv.org/abs/2506.20459}, 
}

@article {AckelsbergBergelson,
    AUTHOR = {Ackelsberg, Ethan and Bergelson, Vitaly},
     TITLE = {Multiple recurrence and popular differences for polynomial
              patterns in rings of integers},
   JOURNAL = {Math. Proc. Cambridge Philos. Soc.},
  FJOURNAL = {Mathematical Proceedings of the Cambridge Philosophical
              Society},
    VOLUME = {176},
      YEAR = {2024},
    NUMBER = {2},
     PAGES = {239--278},
      ISSN = {0305-0041,1469-8064},
   MRCLASS = {37A30 (05D05 11B30 11R09 37A15 37A44)},
  MRNUMBER = {4706769},
MRREVIEWER = {Bryna\ Kra},
       DOI = {10.1017/s030500412300049x},
       URL = {https://doi.org/10.1017/s030500412300049x},
}

@incollection {BLcubic,
    AUTHOR = {Bergelson, V. and Leibman, A.},
     TITLE = {Cubic averages and large intersections},
 BOOKTITLE = {Recent trends in ergodic theory and dynamical systems},
    SERIES = {Contemp. Math.},
    VOLUME = {631},
     PAGES = {5--19},
 PUBLISHER = {Amer. Math. Soc., Providence, RI},
      YEAR = {2015},
      ISBN = {978-1-4704-0931-9},
   MRCLASS = {28D15 (05D10)},
  MRNUMBER = {3330334},
MRREVIEWER = {Robert\ Samuel\ Simon},
       DOI = {10.1090/conm/631/12592},
       URL = {https://doi.org/10.1090/conm/631/12592},
}

@incollection {MR1412607,
    AUTHOR = {Furstenberg, Hillel and Weiss, Benjamin},
     TITLE = {A mean ergodic theorem for
              {$(1/N)\sum^N_{n=1}f(T^nx)g(T^{n^2}x)$}},
 BOOKTITLE = {Convergence in ergodic theory and probability ({C}olumbus,
              {OH}, 1993)},
    SERIES = {Ohio State Univ. Math. Res. Inst. Publ.},
    VOLUME = {5},
     PAGES = {193--227},
 PUBLISHER = {de Gruyter, Berlin},
      YEAR = {1996},
      ISBN = {3-11-014219-8},
   MRCLASS = {28D05},
  MRNUMBER = {1412607},
MRREVIEWER = {Idris\ Assani},
}

@article {leibmanpointwise,
    AUTHOR = {Leibman, A.},
     TITLE = {Pointwise convergence of ergodic averages for polynomial
              sequences of translations on a nilmanifold},
   JOURNAL = {Ergodic Theory Dynam. Systems},
  FJOURNAL = {Ergodic Theory and Dynamical Systems},
    VOLUME = {25},
      YEAR = {2005},
    NUMBER = {1},
     PAGES = {201--213},
      ISSN = {0143-3857,1469-4417},
   MRCLASS = {37A17 (22F30 28D15)},
  MRNUMBER = {2122919},
MRREVIEWER = {Alexander\ Gorodnik},
       DOI = {10.1017/S0143385704000215},
       URL = {https://doi.org/10.1017/S0143385704000215},
}

@incollection {totergodic,
    AUTHOR = {Frantzikinakis, Nikos and Kra, Bryna},
     TITLE = {Polynomial averages converge to the product of integrals},
      NOTE = {Probability in mathematics},
   JOURNAL = {Israel J. Math.},
  FJOURNAL = {Israel Journal of Mathematics},
    VOLUME = {148},
      YEAR = {2005},
     PAGES = {267--276},
      ISSN = {0021-2172,1565-8511},
   MRCLASS = {37A25 (22E25 28D05 37A15 47A35)},
  MRNUMBER = {2191231},
MRREVIEWER = {Anish\ Ghosh},
       DOI = {10.1007/BF02775439},
       URL = {https://doi.org/10.1007/BF02775439},
}

@article {indepoly,
    AUTHOR = {Frantzikinakis, Nikos and Kra, Bryna},
     TITLE = {Ergodic averages for independent polynomials and applications},
   JOURNAL = {J. London Math. Soc. (2)},
  FJOURNAL = {Journal of the London Mathematical Society. Second Series},
    VOLUME = {74},
      YEAR = {2006},
    NUMBER = {1},
     PAGES = {131--142},
      ISSN = {0024-6107,1469-7750},
   MRCLASS = {37A45 (28D05 37A30)},
  MRNUMBER = {2254556},
MRREVIEWER = {Thomas\ Ward},
       DOI = {10.1112/S0024610706023374},
       URL = {https://doi.org/10.1112/S0024610706023374},
}

@article {hostkranil,
    AUTHOR = {Host, Bernard and Kra, Bryna},
     TITLE = {Nonconventional ergodic averages and nilmanifolds},
   JOURNAL = {Ann. of Math. (2)},
  FJOURNAL = {Annals of Mathematics. Second Series},
    VOLUME = {161},
      YEAR = {2005},
    NUMBER = {1},
     PAGES = {397--488},
      ISSN = {0003-486X,1939-8980},
   MRCLASS = {37A05 (28D05)},
  MRNUMBER = {2150389},
MRREVIEWER = {Randall\ McCutcheon},
       DOI = {10.4007/annals.2005.161.397},
       URL = {https://doi.org/10.4007/annals.2005.161.397},
}

@article {ipszemeredi,
    AUTHOR = {Furstenberg, H. and Katznelson, Y.},
     TITLE = {An ergodic {S}zemer\'edi theorem for {IP}-systems and
              combinatorial theory},
   JOURNAL = {J. Analyse Math.},
  FJOURNAL = {Journal d'Analyse Math\'ematique},
    VOLUME = {45},
      YEAR = {1985},
     PAGES = {117--168},
      ISSN = {0021-7670,1565-8538},
   MRCLASS = {28D05 (11B05 20M20)},
  MRNUMBER = {833409},
MRREVIEWER = {B.\ Volkmann},
       DOI = {10.1007/BF02792547},
       URL = {https://doi.org/10.1007/BF02792547},
}

@article {polycharfactor,
    AUTHOR = {Leibman, A.},
     TITLE = {Convergence of multiple ergodic averages along polynomials of
              several variables},
   JOURNAL = {Israel J. Math.},
  FJOURNAL = {Israel Journal of Mathematics},
    VOLUME = {146},
      YEAR = {2005},
     PAGES = {303--315},
      ISSN = {0021-2172,1565-8511},
   MRCLASS = {28D05 (47A35)},
  MRNUMBER = {2151605},
MRREVIEWER = {U.\ Krengel},
       DOI = {10.1007/BF02773538},
       URL = {https://doi.org/10.1007/BF02773538},
}

@article {ziegler,
    AUTHOR = {Ziegler, Tamar},
     TITLE = {Universal characteristic factors and {F}urstenberg averages},
   JOURNAL = {J. Amer. Math. Soc.},
  FJOURNAL = {Journal of the American Mathematical Society},
    VOLUME = {20},
      YEAR = {2007},
    NUMBER = {1},
     PAGES = {53--97},
      ISSN = {0894-0347,1088-6834},
   MRCLASS = {37A30 (28D05 37A25)},
  MRNUMBER = {2257397},
MRREVIEWER = {Randall\ McCutcheon},
       DOI = {10.1090/S0894-0347-06-00532-7},
       URL = {https://doi.org/10.1090/S0894-0347-06-00532-7},
}

@article {largelinear,
    AUTHOR = {Bergelson, Vitaly and Host, Bernard and Kra, Bryna},
     TITLE = {Multiple recurrence and nilsequences},
      NOTE = {With an appendix by Imre Ruzsa},
   JOURNAL = {Invent. Math.},
  FJOURNAL = {Inventiones Mathematicae},
    VOLUME = {160},
      YEAR = {2005},
    NUMBER = {2},
     PAGES = {261--303},
      ISSN = {0020-9910,1432-1297},
   MRCLASS = {37A30 (05D10 28D05 37A05)},
  MRNUMBER = {2138068},
MRREVIEWER = {Randall\ McCutcheon},
       DOI = {10.1007/s00222-004-0428-6},
       URL = {https://doi.org/10.1007/s00222-004-0428-6},
}

@article {polyipszemeredi,
    AUTHOR = {Bergelson, Vitaly and McCutcheon, Randall},
     TITLE = {An ergodic {IP} polynomial {S}zemer\'edi theorem},
   JOURNAL = {Mem. Amer. Math. Soc.},
  FJOURNAL = {Memoirs of the American Mathematical Society},
    VOLUME = {146},
      YEAR = {2000},
    NUMBER = {695},
     PAGES = {viii+106},
      ISSN = {0065-9266,1947-6221},
   MRCLASS = {28D15 (05D10 11B05)},
  MRNUMBER = {1692634},
MRREVIEWER = {Thomas\ Ward},
       DOI = {10.1090/memo/0695},
       URL = {https://doi.org/10.1090/memo/0695},
}

@incollection {MR1784213,
    AUTHOR = {Bergelson, Vitaly and Host, Bernard and McCutcheon, Randall
              and Parreau, Fran\c cois},
     TITLE = {Aspects of uniformity in recurrence},
      NOTE = {Dedicated to the memory of Anzelm Iwanik},
   JOURNAL = {Colloq. Math.},
  FJOURNAL = {Colloquium Mathematicum},
    VOLUME = {84/85},
      YEAR = {2000},
     PAGES = {549--576},
      ISSN = {0010-1354,1730-6302},
   MRCLASS = {37A15},
  MRNUMBER = {1784213},
MRREVIEWER = {Thomas\ Ward},
       DOI = {10.4064/cm-84/85-2-549-576},
       URL = {https://doi.org/10.4064/cm-84/85-2-549-576},
}

@misc{bergelson2026setslargevaluespolynomial,
      title={Sets of large values of polynomial multi-correlation functions}, 
      author={Vitaly Bergelson and Rigoberto Zelada},
      year={2026},
      eprint={2605.23050},
      archivePrefix={arXiv},
      primaryClass={math.DS},
      url={https://arxiv.org/abs/2605.23050}, 
}

@article {ackelsbergshalomrichter,
    AUTHOR = {Ackelsberg, Ethan and Richter, Florian K. and Shalom, Or},
     TITLE = {On the maximal spectral type of nilsystems},
   JOURNAL = {Proc. Amer. Math. Soc. Ser. B},
  FJOURNAL = {Proceedings of the American Mathematical Society. Series B},
    VOLUME = {11},
      YEAR = {2024},
     PAGES = {469--480},
      ISSN = {2330-1511},
   MRCLASS = {37A30 (37A17)},
  MRNUMBER = {4797105},
MRREVIEWER = {Song\ Shao},
       DOI = {10.1090/bproc/229},
       URL = {https://doi.org/10.1090/bproc/229},
}

\end{document}